\documentclass{amsart}

\usepackage[latin1]{inputenc}
\usepackage[english]{babel}
\usepackage{amstext}
\usepackage{latexsym}
\usepackage{amsmath}
\usepackage{amssymb}
\usepackage{amsfonts}
\usepackage{graphicx}
\usepackage{amssymb}
\usepackage{amsthm}
\usepackage{latexsym}
\usepackage{epsfig}
\usepackage{tikz}
\usetikzlibrary{matrix,arrows,decorations.pathmorphing}
\usepackage{graphicx}
\usepackage{enumitem}
\DeclareRobustCommand{\SkipTocEntry}[4]{}

\input xy
\xyoption{all}
\begin{document}

\newtheorem{teo}{Theorem}[section]
\newtheorem{cor}[teo]{Corollary}
\newtheorem{lemma}[teo]{Lemma}
\newtheorem{prop}[teo]{Proposition}
\newtheorem{nota}{Notation}
\newtheorem{defn}[teo]{Definition}
\newtheorem*{defn*}{Definition}
\newtheorem*{teoA}{Theorem A}
\newtheorem*{teoB}{Theorem B} 
\newtheorem*{teoC}{Theorem C}
\newtheorem*{teoD}{Theorem D}
\newtheorem*{teop}{Theorem 1}
\newtheorem*{coro}{Corollary}
\newtheorem{conj}{Conjecture}
\newtheorem*{inv}{Invariance Principle}
\newtheorem*{af}{Claim}
\newtheorem{obs}[teo]{Remark}
\newtheorem{corol}{Corollary}

\title[Lyapunov exponents of partially hyperbolic maps]{Lyapunov exponents of partially hyperbolic volume-preserving maps with 2-dimensional center bundle}

\author{Chao Liang}
\thanks{C.L. has been supported by NNSFC(\#11471344) and CUFE Young Elite Teacher Project (\#QYP1705).}
\address{School of Statistics and Mathematics, Central University of Finance and Economics, Beijing, 100081, China}
\email{chaol@cufe.edu.cn}

\author{Karina Marin}
\thanks{K.M. has been supported by CAPES}
\address{IMPA- Estrada D. Castorina 110, Jardim Bot\^anico, 22460-320 Rio de Janeiro -Brazil.}
\email{kmarin@impa.br}

\author{Jiagang Yang}
\thanks{J.Y. has been partially supported by CNPq, FAPERJ, and PRONEX}
\address{Departamento de Geometria, Instituto de Matem\'atica e Estat\'istica, Universidade Federal Fluminense, Niter\'oi, Brazil}
\email{yangjg@impa.br}

\begin{abstract}
We consider the set of partially hyperbolic symplectic diffeomorphisms which are accessible, have 2-dimensional center bundle and satisfy some pinching and bunching conditions. In this set, we prove that the non-uniformly hyperbolic maps are $C^r$ open and there exists a $C^r$ open and dense subset of continuity points for the center Lyapunov exponents. We also generalize these results to volume-preserving systems. 
\end{abstract}

	
\maketitle
\section{Introduction}

Lyapunov exponents play a key role in understanding the ergodic behavior of a dynamical system. For this reason, it is important to be able to control how they vary with the dynamics and to avoid zero Lyapunov exponents. 

One speaks of \emph{non-uniform hyperbolicity} when all the Lyapunov exponents are different from zero almost everywhere with respect to some preferred invariant measure (for instance, a volume measure). This theory was initiated by Pesin and has many important consequences, most notably: the stable manifold theorem (Pesin \cite{P}), the abundance of periodic points and Smale horseshoes (Katok \cite{K}) and the fact that the fractal dimension of invariant measures is well defined (Ledrappier and Young \cite{LY} and Barreira, Pesin and Schmelling \cite{BPS}).

In the context of partially hyperbolic volume-preserving systems we study the following classical problems: openness of the set of non-uniformly hyperbolic diffeomorphisms and continuity of the center Lyapunov exponents for the $C^r$ topology with $r\geq2$. 

For the $C^1$ topology, Ma\~n\'e \cite{Ma} observed that an area-preserving diffeomorphism is a continuity point for the Lyapunov exponents only if it is either Anosov or all its Lyapunov exponents are equal to zero almost everywhere. His arguments were completed by Bochi \cite{B1} and were extended to arbitrary dimension by Bochi and Viana \cite{B2,BV1}. In particular, Bochi \cite{B2} proved that every partially hyperbolic symplectic diffeomorphism can be $C^1$-approximated by partially hyperbolic diffeomorphisms whose center Lyapunov exponents vanish. This implies that the set of non-uniformly hyperbolic systems is \emph{not} $C^1$ open. Our first result proves that the situation is different when we consider the $C^r$ topology with $r\geq 2$. 

Let $\mathcal{B}^r_{\omega}(M)$ denote the subset of partially hyperbolic symplectic systems which are accessible, have 2-dimensional center bundle and satisfy some pinching and bunching conditions. (All the keywords will be recalled in the next section). This set has two important properties: it is a $C^1$ open set and every $f\in \mathcal{B}^r_{\omega}(M)$ is ergodic. Moreover, by Theorem A in \cite{SW}, if $M=\mathbb{T}^{2d}$ with $d\geq 2$, $\mathcal{B}^r_{\omega}(M)$ is non-empty. 

\begin{teoA} Let $r\geq 2$. The subset of non-uniformly hyperbolic diffeomorphisms in $\mathcal{B}^r_{\omega}(M)$ is $C^r$ open.
\end{teoA}

The continuity of Lyapunov exponents has been extensively studied for the case of linear cocycles. Theorem C in \cite{B1} implies that discontinuity of Lyapunov exponents is typical for continuous $SL(2,\mathbb{R})$-valued cocycles. However, there are some contexts where continuity has been established. Bocker and Viana \cite{BV2} and Malheiro and Viana \cite{MV} proved continuity of Lyapunov exponents for random products of 2-dimensional matrices in the Bernoulli and in the Markov settings. More recently, still for 2-dimensional cocycles, Backes, Brown and Butler \cite{BBB} proved that continuity of Lyapunov exponents holds in the realm of fiber-bunched H\"older cocycles over any hyperbolic systems with local product structure. In higher dimension, continuity of the Lyapunov exponents for i.i.d. random products of matrices has been announced by Avila, Eskin and Viana \cite{AEV}. Our second theorem provides a result about continuity of Lyapunov exponents for diffeomorphisms. 

\begin{teoB} Let $r\geq 2$. There exists a $C^r$ open and dense subset $\mathcal{U}\subset \mathcal{B}^r_{\omega}(M)$ such that every $g\in \mathcal{U}$ is a continuity point for the center Lyapunov exponents in the $C^r$ topology. 
\end{teoB}

Moreover, we are able to extend Theorem A and Theorem B for partially hyperbolic volume-preserving systems.

\section{Preliminaries and Statements}

A diffeomorphism $f\colon M\to M$ of a compact manifold $M$ is \textit{partially hyperbolic} if there exist a nontrivial splitting of the tangent bundle $$TM=E^{s}\oplus E^{c}\oplus E^{u}$$ invariant under the derivative map $Df$, a Riemannian metric $\left\| \cdot \right\|$ on $M$, and positive continuous functions $\chi$, $\widehat{\chi}$, $\nu$, $\widehat{\nu}$, $\gamma$, $\widehat{\gamma}$ with 
$$\chi< \nu < 1 < \widehat{\nu}^{-1} < \widehat{\chi}^{-1} \quad  \text{and} \quad \nu< \gamma < \widehat{\gamma}^{-1}< \widehat{\nu}^{-1},$$ such that for any unit vector $v\in T_{p}M$, 
\begin{equation}\label{ph}
\begin{aligned}
\chi(p)< &\left\| Df_{p}(v) \right\|< \nu(p) \quad \quad \text{if} \; v\in E^{s}(p), \\
\gamma(p) < &\left\| Df_{p}(v) \right\|< \widehat{\gamma}(p)^{-1} \quad \text{if} \; v\in E^{c}(p),\\
\widehat{\nu}(p)^{-1}< &\left\| Df_{p}(v) \right\| < \widehat{\chi}(p)^{-1} \quad \text{if} \; v\in E^{u}(p).
\end{aligned}
\end{equation}

Partial hyperbolicity is a $C^1$ open condition, that is, any diffeomorphism sufficiently $C^1$-close to a partially hyperbolic diffeomorphism is itself partially hyperbolic. Moreover, if $f\colon M\to M$ is partially hyperbolic, then the stable and unstable bundles $E^{s}$ and $E^{u}$ are uniquely integrable and their integral manifolds form two transverse (continuous) foliations $W^{s}$ and $W^{u}$, whose leaves are immersed submanifolds of the same class of differentiability as $f$. These foliations are called the \textit{strong-stable} and \textit{strong-unstable} foliations. They are invariant under $f$, in the sense that
$$f(W^{s}(x))=W^{s}(f(x)) \qquad \text{and}\qquad f(W^{u}(x))=W^{u}(f(x)),$$ where $W^{s}(x)$ and $W^{u}(x)$ denote the leaves of $W^{s}$ and $W^{u}$, respectively, passing through any $x\in M$. 

For more information about partially hyperbolic diffeomorphisms we refer the reader to \cite{BDV,HPS,Sh}.

Given two points $x,y\in M$, $x$ is \textit{accessible} from $y$ if there exists a path that connects $x$ to $y$, which is a concatenation of finitely many subpaths, each of which lies entirely in a single leaf of $W^u$ or a single leaf of $W^s$. We call this type of path, an \textit{su-path}. This defines an equivalence relation and we say that $f$ is \textit{accessible} if $M$ is the unique accessibility class. By the results in \cite{AV2}, accessibility is a $C^1$ open condition among partially hyperbolic diffeomorphisms with 2-dimensional center bundle. We refer the reader to Section 5 of \cite{M} for a detailed outline of the proof. See also Proposition \ref{uniform}. 

\begin{defn}[$\alpha$-pinched]\label{holder} Let $f$ be a partially hyperbolic diffeomorphism and $\alpha>0$. We say that $f$ is $\alpha$-pinched if the functions in Equation (\ref{ph}) satisfy,  
\begin{equation*}
\begin{aligned}
\nu &< \gamma\, \chi^{\alpha} \quad \text{and} \quad \; \widehat{\nu} < \widehat{\gamma}\, \widehat{\chi}^{\alpha}, \\
\nu &< \gamma\, \widehat{\chi}^{\alpha} \quad \text{and} \quad \; \widehat{\nu} < \widehat{\gamma}\, \chi^{\alpha}. 
\end{aligned}
\end{equation*} 
\end{defn}

\begin{defn}[$\alpha$-bunched]\label{bunched} Let $f$ be a partially hyperbolic diffeomorphism and $\alpha>0$. We say that $f$ is $\alpha$-bunched if the functions in Equation (\ref{ph}) satisfy, 
$$\nu^{\alpha} < \gamma \widehat{\gamma} \qquad \text{and} \qquad \widehat{\nu}^{\alpha}< \gamma \widehat{\gamma}.$$ 
\end{defn}

Notice that both conditions, $\alpha$-pinched and $\alpha$-bunched, are $C^1$-open. Moreover, if $f$ is a $C^2$ $\alpha$-pinched diffeomorphism, then $E^c$ is $\alpha$-H\"older. See Section 4 of \cite{PSW2}.

Let $M$ be a symplectic manifold and $\omega$ denote the symplectic form. Then, $\mathit{Diff}^r_{\omega}(M)$ denotes the set of $C^r$ diffeomorphisms preserving $\omega$ and $\mathit{PH}^r_{\omega}(M)$ the subset of $\mathit{Diff}^r_{\omega}(M)$ formed by the partially hyperbolic diffeomorphisms. 

\begin{defn}\label{symple} If $r\geq 2$, then $\mathcal{B}^r_{\omega}(M)$ denotes the subset of $\mathit{PH}^r_{\omega}(M)$ where $f\in \mathcal{B}^r_{\omega}(M)$ if $f$ is accessible, $\alpha$-pinched and $\alpha$-bunched for some $\alpha>0$ and its center bundle is 2-dimensional.
\end{defn}

\begin{obs}\label{ergod} Observe that $\mathcal{B}^r_{\omega}(M)$ is a $C^1$ open set. Moreover, the notion of $\alpha$-bunched defined above implies that the diffeomorphism is center bunched in the sense of Theorem 0.1 of \cite{BW}. Therefore, every diffeomorphism in $\mathcal{B}^r_{\omega}(M)$ is ergodic. 
\end{obs}

If $f$ is a volume-preserving $C^1$ diffeomorphism and $\mu$ denotes the volume induced by a Riemannian metric, then by the Theorem of Oseledets for $\mu$-almost every point $x\in M$, there exist $k(x)\in \mathbb{N}$, real numbers $\widehat{\lambda}_1(f,x)> \cdots > \widehat{\lambda}_{k(x)}(f,x)$ and a splitting $T_{x}M=E^{1}_x\oplus \cdots \oplus E^{k(x)}_x$ of the tangent bundle at $x$, all depending measurably on the point, such that
$$\lim\limits_{n\rightarrow \pm \infty} \frac{1}{n} \text{log} \left\|Df^{n}_x(v)\right\|= \widehat{\lambda}_j(f,x) \quad \text{for all} \; v\in E^j_x \setminus \{0\}. $$
The real numbers $\widehat{\lambda}_j(f,x)$ are the \emph{Lyapunov exponents} of $f$ in the point $x$. 

We say that $f$ is \textit{non-uniformly hyperbolic} if the set of points with non-zero Lyapunov exponents has full measure. 

Let $\lambda_1(f,x)\geq \lambda_2(f,x)\geq \cdots \geq \lambda_d(f,x)$ be the numbers $\widehat{\lambda}_j(f,x)$, each repeated with multiplicity $\dim\, E^j_x$ and written in non-increasing order. If $f$ is ergodic, then the functions $k(x)$ and $\lambda_j(f,x)$ are constants almost everywhere. 

For a partially hyperbolic diffeomorphism $f$, the Lyapunov exponents of $Df\vert E^c$ are called the \emph{center Lyapunov exponents} of $f$. If $\dim E^c=2$, we are going to denote them by $\lambda^c_1(f,x)$ and $\lambda^c_2(f,x)$. Moreover, if $f$ is ergodic and $$\int \log \left|\det(Df_x\vert E^c(x))\right|d\mu=0,\quad \text{then} \quad \lambda^c_1(f)+\lambda^c_2(f)=0.$$ This is always the case for partially hyperbolic symplectic diffeomorphisms, see Lemma 2.5 of \cite{XZ}.

In the following, we give the precise statement of Theorem A and Theorem B.

\begin{teoA} For every $r\geq 2$, the subset of non-uniformly hyperbolic diffeomorphisms in $\mathcal{B}^r_{\omega}(M)$ is $C^r$ open.
\end{teoA}

This result together with Theorem A in \cite{M} implies the following: 

\begin{corol} Let $r\geq 2$ and $f\in \mathcal{B}^{r}_{\omega}(M)$. If the set of periodic points of $f$ is non-empty, then $f$ can be $C^r$-approximated by $C^r$ open subsets of non-uniformly hyperbolic symplectic diffeomorphisms.
\end{corol}

Before enunciating Theorem B, we need to give the definition of continuity points of the center Lyapunov exponents in $\mathcal{B}^r_{\omega}(M)$. 

\begin{defn}\label{contin} We say that a diffeomorphism $f\in \mathcal{B}^r_{\omega}(M)$ is a $C^r$ continuity point for the center Lyapunov exponents if for every $f_k\to f$ in $\mathit{Diff}^r_{\omega}(M)$, $$\lambda^c_1(f_k)\to \lambda^c_1(f).$$
\end{defn}

\begin{obs} Observe that since $f$ is a symplectic diffeomorphism, $\lambda^c_1(f_k)\to \lambda^c_1(f)$ if and only if $\lambda^c_2(f_k)\to \lambda^c_2(f).$
\end{obs}

\begin{teoB} Let $r\geq 2$. There exists a $C^r$ open and dense subset $\mathcal{U}\subset \mathcal{B}^r_{\omega}(M)$ such that every $g\in \mathcal{U}$ is a $C^r$ continuity point for the center Lyapunov exponents. 
\end{teoB}

Theorem \ref{invariant} and Proposition \ref{discont} in Section 4 will be used to prove Theorem A and Theorem B but they also imply the following corollary. 

\begin{defn} We say that a periodic point $p$ with $n_p=per(p)$ is a \emph{quasi-elliptic} periodic point if there exists $1\leq l \leq \dim M/2$ such that $Df^{n_p}_p$ has 2l non-real eigenvalues of norm one and its remaining eigenvalues have norm different from one.
\end{defn} 

\begin{corol} Let $r\geq 2$. Every $f\in \mathcal{B}^r_{\omega}(M)$ having a quasi-elliptic periodic point is a $C^r$ continuity point for the center Lyapunov exponents. 
\end{corol}

Theorem A, Theorem B and Corollary 2 give good evidence about the validity of the following conjecture due to Marcelo Viana. 

\begin{conj} If $r\geq 2$ and $f\in \mathcal{B}^r_{\omega}(M)$, then $f$ is a $C^r$ continuity point for the center Lyapunov exponents.
\end{conj}

\subsection{Volume-preserving case}

Fix $r\geq 2$. Let $\mu$ denote a probability measure in the Lebesgue class, $\mathit{Diff}^r_{\mu}(M)$ the set of volume-preserving $C^r$ diffeomorphisms and $\mathit{PH}^r_{\mu}(M)$ the subset of $\mathit{Diff}^r_{\mu}(M)$ consisting of partially hyperbolic diffeomorphisms. 

In order to generalize the results in the symplectic context to the volume-preserving setting, we need to ask for extra hypotheses in the diffeomorphisms. The key property that we use in the proof of Theorem A and B is that $$\int \log \left|\det(Df_x\vert E^c(x))\right|d\mu=0,$$ for every symplectic diffeomorphism. Therefore, we could consider the subset of $\mathit{PH}^r_{\mu}(M)$ where this condition is satisfied. However, since this set is not $C^r$ open, the results are not such relevant. 







In the following, we consider a $C^1$ open subset of $\mathit{PH}^r_{\mu}(M)$ where it is possible to extend Theorem B.

\begin{defn}\label{pinch} Let $f$ be a partially hyperbolic diffeomorphism with $\dim E^c=2$ and $p$ a periodic point of $f$ with $n_p=per(p)$. We say that $p$ is a \emph{pinching hyperbolic periodic point} if the eigenvalues of $Df^{n_p}_p|E^c(p)$ have different norms and both norms are different from one.
\end{defn} 

\begin{defn}\label{volu} Let $\mathcal{P}^r_{\mu}(M)$ denote the subset of $\mathit{PH}^r_{\mu}(M)$ where $f\in \mathcal{P}^r_{\mu}(M)$ if $f$ is accessible, $\alpha$-pinched and $\alpha$-bunched for some $\alpha>0$, has a pinching hyperbolic periodic point and its center bundle is 2-dimensional.
\end{defn}

\begin{obs} Similar to Remark \ref{ergod}, we have that the set $\mathcal{P}^r_{\mu}(M)$ is $C^1$ open and any $f\in \mathcal{P}^r_{\mu}(M)$ is ergodic. 
\end{obs}

For the set $\mathcal{P}^r_{\mu}(M)$, we can only conclude a version of Theorem A about the simplicity of the center Lyapunov exponents. 

\begin{teoC} Let $r\geq 2$. The set formed by diffeomorphisms having different center Lyapunov exponents is $C^r$ open in $\mathcal{P}^r_{\mu}(M)$.
\end{teoC} 

Now, we state the version of Theorem B and Conjecture 1 for this setting. 

\begin{defn}\label{contin2} We say that a diffeomorphism $f\in \mathcal{P}^r_{\mu}(M)$ is a $C^r$ continuity point for the center Lyapunov exponents if for every $f_k\to f$ in $\mathit{Diff}^r_{\mu}(M)$ and $i\in \{1,2\}$, we have $$\lambda^c_i(f_k)\to \lambda^c_i(f).$$
\end{defn}

\begin{teoD} Let $r\geq 2$. There exists a $C^r$ open and dense subset $\mathcal{U}\subset \mathcal{P}^r_{\mu}(M)$ such that every $g\in \mathcal{U}$ is a $C^r$ continuity point for the center Lyapunov exponents. 
\end{teoD}

\begin{conj} If $r\geq 2$ and $f\in \mathcal{P}^r_{\mu}(M)$, then $f$ is a $C^r$ continuity point for the center Lyapunov exponents. 
\end{conj}

\subsection{Strategy of the proof} 

We are going to discuss the ideas of the proof of Theorem A and B. The proof of Theorem C and D is analogous. 

In Section 4 we prove that if $f$ is a discontinuity point for the center Lyapunov exponents, then the fiber bundle $\mathbb{P}(E^c)$ admits two continuous sections, $x\mapsto a_x$ and $x\mapsto b_x$. If $\mathit{F=Df\vert E^c}$, then these sections are invariant by the cocycle $\mathbb{P}(F)$ and by the invariant stable and unstable holonomies of $\mathbb{P}(F)$. Observe that Corollary 2 is a consequence of this result.  

Theorem A will follow from the fact that the diffeomorphisms having zero center Lyapunov exponents form a closed subset. In order to see this, we take $f_k\to f$ with $\lambda^c_1(f_k)=\lambda^c_2(f_k)=0$ and suppose that $\lambda^c_1(f)\neq \lambda^c_2(f)$. Then, $f$ is a discontinuity point for the center Lyapunov exponents and we have two continuous sections of $\mathit{\mathbb{P}(E^c(f))}$, $x\mapsto a_x$ and $x\mapsto b_x$, with the properties stated above.

Using the Invariance Principle of Avila and Viana, we prove that for every $k$ big enough, there exists a continuous section of $\mathit{\mathbb{P}(E^c(f_k))}$, $x\mapsto a_{k,x}$, which is invariant by the cocycle $\mathbb{P}(F_k)$ and by the invariant stable and unstable holonomies of $\mathbb{P}(F_k)$. Moreover, $a_{k,x}$ is close to $a_x$ or to $b_x$ for every $x\in M$. This will imply that $\lambda^c_1(f_k)\to \lambda^c_1(f)$ or $\lambda^c_2(f_k)\to \lambda^c_2(f)$. Both options contradict the hypothesis of $\lambda^c_1(f_k)=\lambda^c_2(f_k)=0$.

In order to prove Theorem B we will find a diffeomorphism $g$ which is arbitrarily close to $f$ and a neighborhood of $g$, $\mathcal{V}(g)$, such that any diffeomorphism $h\in \mathcal{V}(g)$ does not admit continuous sections of $\mathbb{P}(E^c(h))$ which are invariant by the cocycle $\mathbb{P}(H)$ and by the invariant stable and unstable holonomies of $\mathbb{P}(H)$ for $\mathit{H=Dh\vert E^c(h)}$. We will use the same mechanisms than in \cite{M} to achieve this goal. In Section 7 we give more details about the ideas behind the proof of this theorem. 

Observe that in both cases we are working with sections of $\mathit{\mathbb{P}(E^c(f))}$ and $\mathit{\mathbb{P}(E^c(g))}$ where $g$ is close to $f$. In order to be able to estimate the distance between them, we will consider both fiber bundles as subsets of $\mathit{\mathbb{P}(TM)}$.

\section{Center Derivative Cocycle} 

Let $r\geq 2$, $*\in \{\mu, \omega\}$ and $f\in \mathit{PH}^r_{*}(M)$ with $\dim\, E^c=2$. Recall that $\omega$ denotes a symplectic form and $\mu$ denotes a probability measure in the Lebesgue class. 

We will consider the \emph{center derivative cocycle} associated to $f$, that is, the linear cocycle $F'$ defined by $\mathit{F'=Df\vert E^c}$. Observe that the extremal Lyapunov exponents of $F'$, $\mathit{\lambda_{\pm}(F',x)}$ coincide with the center Lyapunov exponents of $f$.

From now on we fix the Riemannian metric given by Equation (\ref{ph}).

If $\eta\colon M\to \mathbb{R}$ is defined by $\eta(x)=|\det\, F'_x|^{-1/2}$, then we can consider a new cocycle over $f$ by $\mathit{F_x=\eta(x)\cdot F'_x}$. Notice that $\left|\det\, F_x\right|=1$ for every $x\in M$ and the extremal Lyapunov exponents of $F$ satisfy the following,
\begin{equation}\label{lyap0}
\begin{aligned}
\mathit{\lambda_{\pm}(F,x)}&=\mathit{\lambda_{\pm}(F',x)} + \int{ \log \left|\eta(x)\right| \, d\mu}\\
                  &=\mathit{\lambda^c_{1,2}(f,x)} + \int{ \log \left|\eta(x)\right| \, d\mu}
\end{aligned}
\end{equation}

Let $\pi\colon \mathcal{E}\to M$ be a fiber bundle with smooth fibers modeled on some Riemannian manifold $N$. A \emph{smooth cocycle} over $f$ is a continuous transformation $\mathcal{F}\colon \mathcal{E}\to \mathcal{E}$ such that $\mathit{\pi\circ \mathcal{F}=f\circ \pi}$, every $\mathcal{F}_x\colon \mathcal{E}_x\to \mathcal{E}_{f(x)}$ is a $C^1$ diffeomorphism depending continuously on $x$ and the norms of the derivative $D\mathcal{F}_x(\xi)$ and its inverse are bounded.

The \emph{projective cocycle} associated to a linear cocycle $G\colon V \to V$ over $f$ is the smooth cocycle $\mathbb{P}(G)\colon \mathbb{P}(V)\to \mathbb{P}(V)$ whose action on the fibers is given by the projectivization of $G_x\colon V_x \to V_{f(x)}$. 

Notice that $\mathbb{P}(F)=\mathbb{P}(F')$. Since $\dim\,E^c=2$, $\mathbb{P}(F)$ is a cocycle of circle diffeomorphisms over $f$. Moreover, there always exists a $\mathbb{P}(F)$-invariant probability measure $m$ that projects down to $\mu$. This is true because the projective cocycle $\mathbb{P}(F)$ is continuous and the domain $\mathbb{P}(E^c)$ is compact. 

The extremal Lyapunov exponents of $\mathbb{P}(F)$ for $m$ exist and satisfy,
\begin{equation}\label{lyap2}
\begin{aligned}
\mathit{\lambda_{+}(\mathbb{P}(F), x,\xi)}\leq \mathit{\lambda_{+}(F,x)} - \mathit{\lambda_{-}(F,x)}& \quad \text{and}\\
&\mathit{\lambda_{-}(\mathbb{P}(F), x,\xi)}\geq \mathit{\lambda_{-}(F,x)} - \mathit{\lambda_{+}(F,x)},
\end{aligned}
\end{equation}
for every $x\in M$ and $\xi\in \mathit{\mathbb{P}(E^c(x))}$ where they are defined.

\begin{defn}\label{hol} Let $\mathcal{F}\colon \mathcal{E}\to \mathcal{E}$ be a smooth cocycle over $f$. An \emph{invariant stable holonomy} for $\mathcal{F}$ is a family $h^{s}$ of homeomorphisms $h^{s}_{x,y}\colon \mathcal{E}_x\to \mathcal{E}_y$, defined for all $x$ and $y$ in the same strong-stable leaf of $f$ and satisfying
\begin{enumerate} [label=\emph{(\alph*)}]
\item $h^{s}_{y,z}\circ h^{s}_{x,y}= h^{s}_{x,z}$ and $h^{s}_{x,x}=Id$;
\item $\mathcal{F}_y\circ h^{s}_{x,y}= h^{s}_{f(x),f(y)}\circ \mathcal{F}_x$; 
\item $(x,y, \xi)\mapsto h^{s}_{x,y}(\xi)$ is continuous when $(x,y)$ varies in the set of pairs of points in the same local strong-stable leaf;
\item there are $C>0$ and $\beta >0$ such that $h^{s}_{x,y}$ is $(C,\beta)$-H\"older continuous for every $x$ and $y$ in the same local strong-stable leaf.
\end{enumerate}
An \emph{invariant unstable holonomy} for $\mathcal{F}$ can be defined analogously, for pairs of points in the same strong-unstable leaf.
\end{defn}

Condition (c) in Definition \ref{hol} means that given any $\epsilon>0$ and any $(x,y, \xi)$ with $y\in W^s_{loc}(x)$ and $\xi\in \mathcal{E}_x$, there exists $\delta>0$ such that $dist(h^s_{x,y}(\xi), h^s_{x',y'}(\xi'))<\epsilon$ for every $(x', y', \xi')$ with $y'\in W^s_{loc}(x')$, $\xi'\in \mathcal{E}_{x'}$, $dist(x, x')<\delta$, $dist(y, y')<\delta$ and $dist(\xi, \xi')<\delta$. Here to consider the distance between points in different fibers you can think that the fiber bundle has been trivialized in the neighborhoods of $\mathcal{E}_x$ and $\mathcal{E}_y$. 

If $f$ is $\alpha$-pinched and $\alpha$-bunched (Definition \ref{holder} and Definition \ref{bunched}), then $\mathbb{P}(F)$ admits invariant stable and invariant unstable holonomies. This is a consequence of Section 3 of \cite{ASV}. Moreover, if $x$ and $y$ are in the same local strong-stable leaf, then $h^s_{x,y}=\mathbb{P}(H^s_{x,y})$ where $H^s_{x,y}\colon \mathit{E^c(x)}\to \mathit{E^c(y)}$ is a linear isomorphism. Therefore, in the setting we are studying the holonomies $h^s_{x,y}$ are Lipschitz for every $x$ and $y$ in the same local strong-stable leaf. This is also true for the invariant unstable holonomy. 


If $\pi\colon \mathcal{E}\to V$ is a fiber bundle over $M$ and $m$ a probability measure in $\mathcal{E}$ with $\pi_{*}m=\mu$, then there exists a disintegration of $m$ into conditional probabilities $\left\{m_{x} : x\in M \right\}$ along the fibers which is essentially unique, that is, a measurable family of probability measures such that $m_x(\mathcal{E}_x)=1$ for almost every $x\in M$ and $$m(U)=\int m_x(U\cap \mathcal{E}_x) d\mu(x),$$ for every measurable set $U\subset \mathcal{E}$. See \cite{Rok}.

\begin{defn} Let $\mathcal{F}$ be a smooth cocycle over $f$ and $h^s$ an invariant stable holonomy for $\mathcal{F}$. We say a disintegration $\left\{m_{x} : x\in M \right\}$ is \emph{s-invariant} if $$(h^{s}_{x,y})_{*}m_{x}=m_{y} \quad \text{for every} \; x\; \text{and}\; y \; \text{in the same strong-stable leaf.}$$ One speaks of \emph{essential s-invariance} if this holds for $x$ and $y$ in some full $\mu$-measure set. The definitions of \emph{u-invariance} and \emph{essential u-invariance} are analogous. The disintegration is \emph{bi-invariant} if it is both s-invariant and u-invariant and we call it \emph{bi-essentially invariant} if it is both essentially s-invariant and essentially u-invariant.
\end{defn}

\begin{defn} Let $\mathcal{F}$ be a smooth cocycle over $f$ admitting holonomies and $m$ an $\mathcal{F}$-invariant probability measure with $\pi_{*}m=\mu$. If $m$ admits some essentially s-invariant disintegration, then it is called \emph{s-state}. The definition of \emph{u-state} is analogous and we say that $m$ is an \emph{su-state} if it is both an s-state and a u-state. 
\end{defn}

\section{Invariance Principle and discontinuity points}

One of the main tools in the proof of our results is the Invariance Principle, which was first developed by Furstenberg \cite{F} and Ledrappier \cite{L} for random matrices and was extended by Bonatti, G\'omez-Mont, Viana \cite{BGV} to linear cocycles over hyperbolic systems and by Avila, Viana \cite{AV1} and Avila, Santamaria, Viana \cite{ASV} to general (diffeomorphisms) cocycles. In \cite{AV1} the base dynamics is still assumed to be hyperbolic, whereas in \cite{ASV}, it is taken to be partially hyperbolic and volume-preserving. 

In the following we are going to state two theorems from \cite{ASV} which are extensions to our setting of the main result in \cite{L}. The first one gives sufficient conditions for an $\mathcal{F}$-invariant probability measure to be an $s$-state or a $u$-state. 

\begin{teo}[Theorem 4.1 in \cite{ASV}]\label{su}
Let $f$ be a $C^1$ partially hyperbolic diffeomorphism, $\mathcal{F}$ a smooth cocycle over $f$, $\mu$ an $f$-invariant probability measure in the Lebesgue class and $m$ an $\mathcal{F}$-invariant probability measure projecting down to $\mu$.
\begin{enumerate} [label=\emph{(\alph*)}]
\item  If $\mathcal{F}$ admits invariant stable holonomies and $\mathit{\lambda_{-}(\mathcal{F}, x, \xi)}\geq 0$ at $m$-almost every point, then $m$ is an $s$-state. 
\item  If $\mathcal{F}$ admits invariant unstable holonomies and $\mathit{\lambda_{+}(\mathcal{F}, x, \xi)}\leq 0$ at $m$-almost every point, then $m$ is a $u$-state. 
\end{enumerate}
\end{teo}

The next theorem will allow us to conclude that an $su$-state is bi-invariant.

\begin{teo}[Theorem D in \cite{ASV}]\label{invariant}
Let $f$ be a $C^2$ partially hyperbolic center bunched diffeomorphism, $\mathcal{F}$ a smooth cocycle over $f$ admitting holonomies, $\mu$ an $f$-invariant probability measure in the Lebesgue class and $m$ an $\mathcal{F}$-invariant probability measure projecting down to $\mu$. 
If $m$ is an $su$-state, then $m$ admits a disintegration $\left\{m_{x} : x\in M \right\}$ along the fibers such that 
\begin{enumerate} [label=\emph{(\alph*)}]
\item  the disintegration is bi-invariant over a full measure bi-saturated set $M_{\mathcal{F}}\subset M$;
\item  if $f$ is accessible, then $M_{\mathcal{F}}=M$ and the conditional probabilities $m_{x}$ depend continuously on the base point $x\in M$, relative to the $\text{weak}^{*}$ topology.
\end{enumerate} 
\end{teo}

\begin{obs}\label{finv} Observe that if $m$ is an $\mathcal{F}$-invariant probability measure which admits a continuous disintegration $\{m_x: x\in M\}$, then $(\mathcal{F}_x)_{*}m_{x}=m_{f(x)}$ for every $x\in M$.
\end{obs}

In the last part of this section, we are going to prove some results that will be used for the proof of the theorems in both settings, the symplectic and the volume-preserving one. In order to simplify the statements of these results, we define the following set:

\begin{defn}\label{gral} If $r\geq 2$ and $*\in \{\mu, \omega\}$, then $\mathcal{B}^r_{*}(M)$ denotes the subset of $\mathit{PH}^r_{*}(M)$ where $f\in \mathcal{B}^r_{*}(M)$ if $f$ is accessible, $\alpha$-pinched and $\alpha$-bunched for some $\alpha>0$ and its center bundle is 2-dimensional.
\end{defn}

Recall that $\omega$ denotes a symplectic form and $\mu$ denotes a probability measure in the Lebesgue class. Observe that this is a $C^1$ open set and every $f\in \mathcal{B}^r_{*}(M)$ is ergodic. In fact, Definition \ref{gral} coincides with Definition \ref{symple} when $*=\omega$ and the set $\mathcal{P}^r_{\mu}(M)$ in Definition \ref{volu} is the open subset of $\mathcal{B}^r_{\mu}(M)$ where $f\in \mathcal{P}^r_{\mu}(M)$ if $f$ has a pinching hyperbolic periodic point. 

Let $f\in \mathcal{B}^r_{*}(M)$ and fix the Riemannian metric given by Equation (\ref{ph}). Let $\eta(x)=|\det\, Df_x\vert E^c(x)|^{-1/2}$ for every $x\in M$. We consider the linear cocycle $F$ defined by $\mathit{F_x=\eta(x)\cdot Df_x\vert E^c(x)}$. The relation between the extremal Lyapunov exponents of $F$ and the center Lyapunov exponents of $f$ is given in Equation (\ref{lyap0}). Then, the extremal Lyapunov exponents of $F$ are constant almost everywhere and satisfy $\lambda_{+}(F)+\lambda_{-}(F)=0$.

We will study the following two cases separately: $$\lambda_{+}(F)=\lambda_{-}(F)\quad \text{and} \quad \lambda_{+}(F)\neq \lambda_{-}(F) .$$

\subsection{Zero Lyapunov exponents.}

From Equation (\ref{lyap2}), if $\lambda_{+}(F)=\lambda_{-}(F)$, then for every $\mathbb{P}(F)$-invariant probability measure $m$ projecting down to $\mu$, $\mathit{\lambda_{\pm}(\mathbb{P}(F), x, \xi)}=0$ for $m$-almost every point. The following theorem is a direct consequence of Theorem \ref{su} and Theorem \ref{invariant} and will allow us to obtain results for $\mathbb{P}(F)$ in this case.

\begin{inv} [Theorem C in \cite{ASV}] Let $f\colon M\to M$ be a $C^{2}$ partially hyperbolic, volume-preserving, center bunched diffeomorphism and $\mu$ be an invariant probability in the Lebesgue class. Let $\mathcal{F}$ be a smooth cocycle over $f$ admitting holonomies and $m$ be an $\mathcal{F}$-invariant probability measure projecting down to $\mu$. Suppose that $\mathit{\lambda_{-}(\mathcal{F}, x,\xi)}=\mathit{\lambda_{+}(\mathcal{F}, x,\xi)}=0$ at $m$-almost every point.

Then, $m$ admits a disintegration $\left\{m_{x} : x\in M \right\}$ along the fibers such that 
\begin{enumerate} [label=\emph{(\alph*)}]
\item  the disintegration is bi-invariant over a full measure bi-saturated set $M_{\mathcal{F}}\subset M$;
\item  if $f$ is accessible, then $M_{\mathcal{F}}=M$ and the conditional probabilities $m_{x}$ depend continuously on the base point $x\in M$, relative to the $\text{weak}^{*}$ topology.
\end{enumerate}
\end{inv}

\begin{cor}\label{zeroF} Let $r\geq 2$ and $*\in \{\mu,\omega\}$. Suppose $f\in \mathcal{B}^r_{*}(M)$, $\mathit{F=\eta\cdot Df\vert E^c}$ and $\lambda_{+}(F)=\lambda_{-}(F)$. If $m$ is an $\mathbb{P}(F)$-invariant probability measure projec\-ting down to $\mu$, then $m$ admits a disintegration $\left\{m_{x} : x\in M \right\}$ along the fibers which is bi-invariant. Moreover, the conditional probabilities $m_{x}$ depend continuously on the base point $x\in M$, relative to the $\text{weak}^{*}$ topology. 
\end{cor}

\subsection{Non-zero center Lyapunov exponents.}

Now we study the case $\lambda_{+}(F)\neq \lambda_{-}(F)$. There are classical versions of Proposition \ref{nozero} and Proposition \ref{discont} for cocycles with a fixed base. See for example Section 6 of \cite{AV1}. Here we extend those standard results to the case where the cocycle depends on the base diffeomorphism. 

\begin{prop}\label{nozero} Let $r\geq 2$ and $*\in \{\mu, \omega\}$. If $f\in \mathcal{B}^r_{*}(M)$, $\mathit{F=\eta\cdot Df\vert E^c}$ and $\lambda_{+}(F)> 0 > \lambda_{-}(F)$, then there exist two $\mathbb{P}(F)$-invariant probability measures projecting down to $\mu$ denoted by $m^{+}$ and $m^{-}$, which are a u-state and an s-state respectively. Moreover, if $m$ is any $\mathbb{P}(F)$-invariant probability measure projecting down to $\mu$, then there exists $t\in [0,1]$ such that $m=t\, m^{+} + (1-t)\, m^{-}$.
\end{prop} 
\begin{proof}

Since $\lambda_{+}(F) > 0 > \lambda_{-}(F)$, by Equation (\ref{lyap0}), $\lambda^c_1(f)\neq \lambda^c_2(f)$. Let $E^c(x)=E^{+}_x\oplus E^{-}_x$ denote the decomposition given by the Theorem of Oseledets for $\mu$-almost every $x\in M$. Then, we can define two probability measures in $\mathbb{P}(E^c)$, $$m^{+}=\int \delta_{\mathbb{P}(E^{+}_x)}\, d\mu \quad \text{and} \quad m^{-}=\int \delta_{\mathbb{P}(E^{-}_x)}\, d\mu.$$ 
Notice that $m^{+}$ and $m^{-}$ are $\mathbb{P}(F)$-invariant probability measures and project down to $\mu$. Moreover, we can calculate its Lyapunov exponents and we obtain the following: $$\lambda_{\pm}(\mathbb{P}(F), m^{+})=-2\, \lambda_{+}(F) \quad \text{and} \quad \lambda_{\pm}(\mathbb{P}(F), m^{-})=-2\, \lambda_{-}(F).$$ Therefore, by Theorem \ref{su} (b) we conclude that $m^{+}$ is a $u$-state and by Theorem \ref{su} (a) that $m^{-}$ is an $s$-state. 

The proof of the second part of the proposition is a consequence of the fact that every compact subset of $\mathbb{P}(E^c)$ disjoint from $\{\mathbb{P}(E^{+}), \mathbb{P}(E^{-})\}$ accumulates on $\mathbb{P}(E^{+})$ in the future and on $\mathbb{P}(E^{-})$ in the past. 
\end{proof}

For $(x,v)\in \mathbb{P}(E^c)$, let $\Phi(x,v)=\log \left\|F_x(v)\right\|$. The next lemma is a classical result for linear cocycles. We refer the reader to Section 6 of \cite{V}.

\begin{lemma}\label{fust} The exponent $\lambda_{+}(F)$ coincides with the maximum of $\int{\Phi(x,v)\, dm}$ over all $\mathbb{P}(F)$-invariant measures $m$ projecting down to $\mu$. Moreover, when $\lambda_{+}(F)>0$ the probability measure $m^{+}$, defined in Proposition \ref{nozero}, realizes the maximum.
\end{lemma}

The following result gives a characterization of the discontinuity points of the center Lyapunov exponents. 

\begin{prop}\label{discont} Let $r\geq 2$ and $*\in \{\mu, \omega\}$. Suppose $f\in \mathcal{B}^r_{*}(M)$ is a $C^r$ discontinuity point for the center Lyapunov exponents and $\mathit{F=\eta\cdot Df\vert E^c}$. Then, every $\mathbb{P}(F)$-invariant probability measure $m$ projecting down to $\mu$ is an su-state.
\end{prop}
\begin{proof}
By the hypotheses, there exists a sequence $f_k\to f$ in $\mathit{Diff}_{*}^r(M)$ such that $\lambda^c_i(f_k)$ does not converges to $\lambda^c_i(f)$ for some $i\in \{1,2\}$. 

Since the functions $f\mapsto \lambda^c_1(f)$ and $f\mapsto \lambda^c_2(f)$ are upper semi-continuous and lower semi-continuous, respectively, the discontinuity of $\lambda^c_{i}(f)$ for some $i\in \{1,2\}$ implies that $\lambda^c_1(f)\neq \lambda^c_2(f)$. Therefore, by Equation (\ref{lyap0}), $\lambda_{+}(F)\neq \lambda_{-}(F)$ and since $\lambda_{+}(F)+\lambda_{-}(F)=0$, we have $\lambda_{+}(F)> 0 > \lambda_{-}(F)$. 

Let $m^{+}$ and $m^{-}$ be given by Proposition \ref{nozero}.

Consider now the cocycle $F_k=\eta_k \cdot Df_k\vert E^c(f_k)$ associated to $f_k$ and let $m_k$ be an ergodic probability measure for $\mathbb{P}(F_k)$ which realizes the maximum in Lemma \ref{fust}. Then, $\lambda_{+}(F_k)=\int{ \Phi_k(x,v)\, d\, m_k}$ and $\lambda_{+}(F_k)$ does not converges to $\lambda{+}(F)$. Moreover, by similarity of the Lyapunov exponents, $\lambda_{-}(F_k)$ does not converges to $\lambda_{-}(F)$.

Observe that $m_k$ is a $u$-state for every $k\in \mathbb{N}$. Moreover, there exist a subsequence $k_j$ and a measure $m$ in $\mathit{\mathbb{P}(TM)}$ such that $m_{k_j}\rightarrow m$ in the weak$^*$ topology. The limit measure $m$ has the following properties: 
\begin{enumerate}[label=(\alph*)]
\item $\mathit{supp}\; m \subset \mathit{\mathbb{P}(E^c(f))}$,
\item $m$ projects down to $\mu$,
\item $m$ is $\mathbb{P}(F)$-invariant,
\item $m$ is a $u$-state.
\end{enumerate}


Moreover, since $f_k\to f$ we have $\int{ \Phi_{k_j}(x,v)\, d\, m_{k_j}}\to \int{ \Phi(x,v)\, d\, m}$. On the other hand, since $\lambda_{+}(F_k)$ does not converges to $\lambda_{+}(F)$, $$\lim \limits_{k_j} \int{ \Phi_{k_j}(x,v)\, d\, m_{k_j}} < \lambda_{+}(F).$$ 

These properties allow us to conclude that $m$ is an $\mathbb{P}(F)$-invariant probability measure projecting down to $\mu$ which is a $u$-state and it is different from $m^{+}$. Therefore, by Proposition \ref{nozero}, there exists $t\neq 1$ such that $m= t\, m^{+} + (1-t)\, m^{-}$. Now, we can write $m^{-}=\frac{m- t\, m^{+}}{(1-t)}$. Moreover, we know that $m^{+}$ is a u-state and $m^{-}$ an s-state. This implies, $m^{-}$ is an $su$-state. 

Using an analogous result of Lemma \ref{fust} for $\lambda_{-}(F)$ and repeating the argument, we conclude that $m^{+}$ is also an $su$-state. 
\end{proof}

This proposition together with Theorem \ref{invariant} imply the following corollary which contains Corollary 2.

\begin{corol} Let $r\geq 2$ and $*\in \{\mu, \omega\}$. Suppose $f\in \mathcal{B}^r_{*}(M)$, $F=\eta\cdot Df\vert E^c$ and one of the following is satisfied: 
\begin{enumerate}[label=\emph{(\alph*)}]
\item There exists $p\in Per(f)$ with $per(p)=n_p$ such that $\mathbb{P}(F^{n_p}_p)\colon \mathit{\mathbb{P}(E^c(p))}\to \mathit{\mathbb{P}(E^c(p))}$ has no fixed points, or 
\item There exist $x\in M$ and an $su$-path $\gamma$ from $x$ to itself such that the holonomy for $\mathbb{P}(F)$ defined by $\gamma$, $h\colon \mathit{\mathbb{P}(E^c(x))}\to \mathit{\mathbb{P}(E^c(x))}$, has no fixed points.
\end{enumerate}
Then, $f$ is a $C^r$ continuity point for the center Lyapunov exponents. 
\end{corol}

It is easy to see that Propositions \ref{nozero} and \ref{discont} are also valid in the context of Section 8 of \cite{ASV}. Therefore, the next corollary follows: 

\begin{corol} Let $f\colon M\to M$ be a $C^2$ partially hyperbolic, volume-preserving, center bunched, accessible diffeomorphism and $\mu$ an invariant probability measure in the Lebesgue class.
If $\mathit{\mathcal{G}^{r,\alpha}(V,f)}$ denotes the set of $C^{r,\alpha}$ fiber bunched linear cocycles $F\colon V\to V$ over $f$ with fiber modeled by $\mathbb{R}^2$, then there exists an open and dense subset $\mathcal{U}\subset \mathit{\mathcal{G}^{r,\alpha}(V,f)}$ such that every $F\in \mathcal{U}$ is a continuity point for $F\to \lambda_{\pm}(F)$.
\end{corol} 

\section {Accessibility and continuity of holonomies}

In this section we are going to state some theorems about accessibility which already appear in \cite{AV2, M}. These results will allow us to obtain estimations for the variation of the holonomies associated to the center derivative cocycle of $f$ when we perturb the diffeomorphism.

We recall the following definitions: Given two points $x,y\in M$, $x$ is \textit{accessible} from $y$ if there exists a path that connects $x$ to $y$, which is a concatenation of finitely many subpaths, each of which lies entirely in a single leaf of $W^u$ or a single leaf of $W^s$. We call this type of path, an \textit{su-path}. This defines an equivalence relation and we say that $f$ is \textit{accessible} if $M$ is the unique accessibility class.

Given $\gamma$ an $su$-path, there exist finitely many points $z_i$ which are defined by the extremal points of the finitely many subpaths that compose the $su$-path. That is $z_i\in W^{*}(z_{i+1})$ for every $i\in \{0,...,n-1\}$ and $*\in \{s,u\}$. The points $z_i$ are called the \emph{nodes} of the $su$-path. We are going to use the following notation: $\gamma=[z_0,z_1,...,z_n]$.

If the partially hyperbolic diffeomorphism $f$ has 2-dimensional center bundle, then we can apply the results in \cite{AV2} in order to have the following proposition. See also the proof of Corollary 5.8 in \cite{M} for more details about this result and its proof. 

\begin{prop}\label{uniform} Let $f$ be a $C^1$ partially hyperbolic accessible diffeomorphism with 2-dimensional center bundle. Then, there exist $N\in \mathbb{N}$ and a neighborhood of $f$ in the $C^1$ topology, $\mathcal{V}(f)$, such that for any $x,y\in M$ and $g\in \mathcal{V}(f)$ there exists an $su$-path for $g$ joining $x$ to $y$ with at most $N$ nodes and the distance between the nodes bounded by $N$. 
\end{prop}

The next results give two refinements of the above proposition. We consider a sequence $f_k\to f$ in the $C^1$ topology and obtain some kind of continuity for $su$-paths under the variation of the diffeomorphism. 

The first one is a simple consequence of the fact that $W^s(x,f)$ and $W^u(x,f)$ vary continuously with the point $x$ and the diffeomorphism $f$.

\begin{prop}\label{a cont}
Let $f$ be a $C^1$ partially hyperbolic accessible diffeomorphism with 2-dimensional center bundle. For every $\epsilon>0$, every $x\in M$, $x_k\to x$ and every sequence $f_k\to f$ in the $C^1$ topology, there exists $K\in \mathbb{N}$ such that for every $k\geq K$ and every $su$-path for $f_k$ given by Proposition \ref{uniform}, $\gamma_k=[z_0(f_k),...,z_N(f_k)]$, with $z_0(f_k)=x_k$, there exists an $su$-path for $f$, $\gamma=[z_0,...,z_N]$ with $z_0=x$ such that for every $i\in \{0,...,N\}$
$$dist (z_i, z_i(f_k))< \epsilon.$$
\end{prop}

Notice that in the proposition above the only information that we have about the final node of $\gamma$ is that it is $\epsilon$-close to the final node of $\gamma_k$. The next proposition deals with the case where we need to fix the initial and final points. In this case we also obtain some continuity of the $su$-path but we need to consider a subsequence of $f_k$. More precisely, 

\begin{prop}[Proposition 5.7 and Corollary 5.8 in \cite{M}]\label{continuity-su}
Let $f$ be a $C^1$ partially hyperbolic accessible diffeomorphism with 2-dimensional center bundle. Then, for every $x,y\in M$, $x_k\to x$, $y_k\to y$ and every sequence $f_k\to f$ in the $C^1$ topology, there exist a subsequence $k_j$, $su$-paths for $f_{k_j}$ denoted by $\gamma_{k_j}$ and an $su$-path for $f$ denoted by $\gamma$ satisfying the following: 
\begin{enumerate}[label=\emph{(\alph*)}] 
\item $\gamma_{k_j}=[z_0(f_{k_j}), ... ,z_n(f_{k_j})]$ joins $x_{k_j}$ to $y_{k_j}$,
\item $\gamma=[z_0, ..., z_n]$ joins $x$ to $y$ and 
\item for every $\epsilon>0$ there exists $K\in \mathbb{N}$ such that for every $k_j\geq K$, $$dist (z_i, z_i(f_{k_j}))< \epsilon$$ for every $i\in \{0,...,n\}$.
\end{enumerate}
Moreover, if $N$ is given by Proposition \ref{uniform}, then the $su$-paths in \emph{(a)} and \emph{(b)} have at most $N$ nodes and the distance between the nodes of each $su$-path is bounded by $N$. 
\end{prop}

Although the proof of this result in \cite{M} is done for the case of $x_k=x$, it is easy to see that the same proof can be extended to the case we need here. 

Recall that $\omega$ denotes a symplectic form, $\mu$ denotes a probability measure in the Lebesgue class and by Definition \ref{gral} if $r\geq 2$ and $*\in \{\mu, \omega\}$, then $\mathcal{B}^r_{*}(M)$ denotes the subset of $\mathit{PH}^r_{*}(M)$ where $f\in \mathcal{B}^r_{*}(M)$ if $f$ is accessible, $\alpha$-pinched and $\alpha$-bunched for some $\alpha>0$ and its center bundle is 2-dimensional.

Let $f\in \mathcal{B}^r_{*}(M)$ and fix the Riemannian metric given by Equation (\ref{ph}). If $\eta(x)=|\det\, Df_x\vert E^c(x)|^{-1/2}$ for every $x\in M$, we consider the linear cocycle $\mathit{F=\eta\cdot Df\vert E^c}$ over $f$ and denote $\mathbb{P}(F)$ its projectivization. Then, we have invariant stable and unstable holonomies associated to $\mathbb{P}(F)$. More precisely, for every $y\in W^s(x)$ there exists an homeomorphism $h_{x,y}^s\colon\mathit{\mathbb{P}(E^c(x))}\to \mathit{\mathbb{P}(E^c(y))}$ satisfying the properties in Definition \ref{hol}. Analogously, for every $y\in W^u(x)$ we have an homeomorphism $h^u_{x,y}$. 

Given an $su$-path $\gamma=[z_0,z_1,...,z_n]$, we define the holonomy associated to it by $h_{\gamma}=h_{z_{n-1}}\circ ... \circ h_{z_0}$ where $h_{z_i}=h^{s}_{z_{i}, z_{i+1}}$ if $z_i\in W^s(z_{i+1})$ and $h_{z_i}=h^{u}_{z_{i}, z_{i+1}}$ if $z_i\in W^u(z_{i+1})$.

In Proposition 3.4 and Corollary 3.5 of \cite{M} it is proved that there exist a $C^2$-neighborhood of $f$, $\mathcal{U}(f)$, in which the holonomy for $\gamma$ varies continuously with the diffeomorphism. This is, if $g\in \mathcal{U}(f)$ is $C^1$-close enough to $f$ and $\gamma_g$ is an $su$-path for $g$ whose nodes are close enough to the nodes of $\gamma$, then the respective holonomies are close. We are going to precise this statement in the following two results which are corollaries of the propositions above. 

When we refer to the distance between a point $a\in \mathit{\mathbb{P}(E^c(x,f))}$ and a point $a(g)\in \mathit{\mathbb{P}(E^c(y,g))}$, we are considering both as elements in $\mathit{\mathbb{P}(TM)}$. The distance between points in different fibers is defined using parallel transport. More precisely, for every $x,y\in M$ close enough, denote $\pi_{x,y}:T_{x}M\longrightarrow T_{y}M$ the parallel transport along $\chi$, where $\chi$ is the geodesic satisfying $dist(x,y)=\text{length}(\chi)$. Then, given two points $(x,v)$ and $(y,w)$ in $\mathit{\mathbb{P}(TM)}$ define $$d((x,v),(y,w))=dist(x,y) + \angle(\pi_{x,y}(v), w).$$

From now on we fix the $C^2$ neighborhood $\mathcal{U}(f)$ where Proposition 3.4 and Corollary 3.5 of \cite{M} hold.

The first corollary is a consequence of Proposition \ref{a cont}.

\begin{cor}\label{hol a cont} Let $*\in \{\mu, \omega\}$, $r\geq 2$ and $f\in \mathcal{B}^r_{*}(M)$. There exists $C>0$ such that for every $x\in M$, $x_k\to x$ and every sequence $f_k\to f$ in the $C^1$ topology with $f_k\in \mathcal{U}(f)$ for every $k\in \mathbb{N}$, there exists $K\in \mathbb{N}$ such that for every $k\geq K$ and every $su$-path for $f_k$ given by Proposition \ref{uniform}, $\gamma_k=[z_0(f_k),...,z_N(f_k)]$, with $z_0(f_k)=x_k$, the $su$-path for $f$, $\gamma=[z_0,...,z_N]$, given by Proposition \ref{a cont} satisfies the following estimation for the holonomies of $\mathbb{P}(F)$ and $\mathbb{P}(F_k)$,
$$d(h_{\gamma}(c), h_{\gamma_{k}}(d))\leq \psi(k) + C \,d(c,d) \quad \forall\; c\in \mathit{\mathbb{P}(E^c(x,f))},\; d\in \mathit{\mathbb{P}(E^c(x_{k},f_{k}))},$$
where $\psi(k)$ goes to zero as $k$ goes to $\infty$.
\end{cor}

The second corollary is a consequence of Proposition \ref{continuity-su}.

\begin{cor}\label{bound-sequence} Let $*\in \{\mu, \omega\}$, $r\geq 2$ and $f\in \mathcal{B}^r_{*}(M)$. There exists $C>0$ such that for every $x,y\in M$, $x_k\to x$, $y_k\to y$ and every sequence $f_k\to f$ in the $C^1$ topology with $f_k\in \mathcal{U}(f)$ for every $k\in \mathbb{N}$, the $su$-paths given by Proposition \ref{continuity-su}, denoted by $\gamma_{k_j}$ and $\gamma$, can be taken to satisfy the following estimation for the holonomies defined by them,
$$d(h_{\gamma}(c), h_{\gamma_{k_j}}(d))\leq \psi(k_j) + C \,d(c,d) \quad \forall\; c\in \mathit{\mathbb{P}(E^c(x,f))},\; d\in \mathit{\mathbb{P}(E^c(x_{k_j}, f_{k_j}))},$$
where $\psi(k_j)$ goes to zero as $k_j$ goes to $\infty$.
\end{cor}

\section{Proof of Theorem A and C}

Recall that $\omega$ denotes a symplectic form, $\mu$ denotes a probability measure in the Lebesgue class and by Definition \ref{gral} if $r\geq 2$ and $*\in \{\mu, \omega\}$, then $\mathcal{B}^r_{*}(M)$ denotes the subset of $\mathit{PH}^r_{*}(M)$ where $f\in \mathcal{B}^r_{*}(M)$ if $f$ is accessible, $\alpha$-pinched and $\alpha$-bunched for some $\alpha>0$ and its center bundle is 2-dimensional. Moreover, by Definition \ref{pinch} we say that a periodic point $p$ is a pinching hyperbolic periodic point if the eigenvalues of $Df^{n_p}_p|E^c(p)$ have different norms and both norms are different from one.

We will prove the following theorem. 

\begin{teo}\label{open} Let $r\geq 2$ and $*\in \{\mu, \omega\}$. Suppose $f_k\to f$ in $\mathit{Diff}^r_{*}(M)$, $f\in \mathcal{B}^r_{*}(M)$ and $f$ has a pinching hyperbolic periodic point. If $\lambda^c_1(f_k)=\lambda^c_2(f_k)$ for every $k\in \mathbb{N}$, then $\lambda^c_1(f)=\lambda^c_2(f)$.
\end{teo}

It is clear that this theorem will imply Theorem C. Moreover, suppose that Theorem A is not true. Therefore, there exists $f\in \mathcal{B}^r_{\omega}(M)$ such that $f$ is a non-uniformly hyperbolic diffeomorphism and there exists a sequence $f_k\to f$ in $\mathit{Diff}^r_{\omega}(M)$ with $\lambda^c_1(f_k)=\lambda^c_2(f_k)$ for every $k\in \mathbb{N}$. By Theorem 4.2 of \cite{K}, there exists a hyperbolic periodic point $p$ for $f$ which is in fact a pinching hyperbolic periodic point, because $f$ is a symplectic diffeomorphism. This contradicts Theorem \ref{open} and therefore Theorem A has to be true.  

\vspace{0.3cm}
\emph{Proof of Theorem \ref{open}.}
Let $r\geq 2$ and $*\in \{\mu, \omega\}$. Let $f_k\to f$ in $\mathit{Diff}^r_{*}(M)$ and $\lambda^c_1(f_k)=\lambda^c_2(f_k)$ for every $k\in \mathbb{N}$. Assume that $f\in \mathcal{B}^r_{*}(M)$, $f$ has a pinching hyperbolic point and $\lambda^c_1(f)\neq \lambda^c_2(f)$. 

By the hypotheses, $f$ is a $C^r$ discontinuity point for the center Lyapunov exponents. Moreover, if $\mathit{F=\eta\cdot Df\vert E^c}$, by Equation (\ref{lyap0}), $\lambda_{+}(F)>0>\lambda_{-}(F)$. See the argument in the first paragraph of the proof of Proposition \ref{discont}.

We can apply Proposition \ref{nozero} and Proposition \ref{discont}. Then, there exist two $\mathbb{P}(F)$-invariant probability measures projecting down to $\mu$, $m^{+}$ and $m^{-}$, which are $su$-states.
 
By Theorem \ref{invariant} we know that both $m^{+}$ and $m^{-}$ admit disintegrations which are bi-invariant and their conditional probabilities depend continuously on the base point $x\in M$, relative to the weak$^{*}$ topology. We are going to denote these disintegrations by $\{m^{+}_x: x\in M\}$ and $\{m^{-}_x: x\in M\}$, respectively. Observe that $m^{+}_x=\delta_{\mathbb{P}(E^{+}_x)}\;$ and $m^{-}_x=\delta_{\mathbb{P}(E^{-}_x)}\;$ for $\mu$-almost every $x\in M$. 

Define $$\mathit{M^{+}=\mathit{supp}\,m^{+}=\{(x, \mathit{supp}\, m^{+}_x) : x\in M\}},$$ $$\mathit{M^{-}=\mathit{supp}\, m^{-}=\{(x, \mathit{supp}\, m^{-}_x) : x\in M\}.}$$ Then, $M^{+}\cap M^{-}=\emptyset$. Since the disintegrations are bi-invariant and $f$ is accessible, if there were some point $(x,v) \in M^{+}\cap M^{-}$, it would imply that $M^{+}=M^{-}$ which is a contradiction. 

Since $M^{+}$ and $M^{-}$ are two disjoint compact sets of $\mathit{\mathbb{P}(TM)}$, there exists $\epsilon>0$ such that $$B_{\epsilon}(M^{+})\cap B_{\epsilon}(M^{-})=\emptyset.$$

Let $p$ be a pinching hyperbolic periodic point for $f$ and $n_p=per(p)$. Define $a,b\in \mathit{\mathbb{P}(E^c(p,f))}$ as $a=\mathbb{P}(E^1)$ and $b=\mathbb{P}(E^2)$, where $E^1$ and $E^2$ are the subspaces of $\mathit{E^c(p,f)}$ associated to the eigenvalues of $\mathit{Df^{n_p}_p\vert E^c(p,f)}$. 

For every $k\in \mathbb{N}$, let $F_k=\eta_k\cdot Df_k\vert E^c(f_k)$ and $m_k$ be any ergodic probability measure for $\mathbb{P}(F_k)$. By Equation (\ref{lyap0}), $\lambda_{+}(F_k)=\lambda_{-}(F_k)$. If $k$ is big enough, $f_k\in \mathcal{B}^r_{*}(M)$ and we can apply Corollary \ref{zeroF}. This implies that there exists a disintegration $\{m_{k,x}: x\in M\}$ which is bi-invariant and $m_{k,x}$ depends continuously on the base point $x\in M$.

Moreover, if $k$ is big enough, there exists a pinching hyperbolic periodic point for $f_k$ that we denote $p(f_k)$, such that $p(f_k)\to p$ as $k\to \infty$. If $a(f_k)=\mathit{\mathbb{P}(E^1(f_k))}$ and $b(f_k)=\mathit{\mathbb{P}(E^2(f_k))}$ where $E^1(f_k)$ and $E^2(f_k)$ are the subspaces of $\mathit{E^c(p(f_k), f_k)}$ associated to the eigenvalues of $Df_{k}^{n_p}\vert E^c(p(f_k), f_k)$, then $a(f_k)\to a$ and $b(f_k)\to b$ when $k\to \infty$.   

Since the measure $m_k$ is $\mathbb{P}(F_k)$-invariant, by Remark \ref{finv}, we have $$\mathit{supp}\, m_{k, p(f_k)}\subset \{a(f_k), b(f_k)\}.$$ We are going to prove that there exists a subsequence $k_j$ such that $$m_{k_j,p(f_{k_j})}=\delta_{a(f_{k_j})}\quad \text{or}\quad m_{k_j,p(f_{k_j})}=\delta_{b(f_{k_j})},$$ for every $j\in \mathbb{N}$. 

In order to prove the statement above, suppose there exists $K\in \mathbb{N}$ such that for every $k\geq K$ there exists $t\in (0,1)$ such that $m_{k, p(f_k)}= t \delta_{a(f_k)} + (1-t) \delta_{b(f_k)}$. 

Fix $x\in M$. By Proposition \ref{uniform}, there exists $\gamma_{k}=\gamma(f_k,p(f_k),x)$ an $su$-path for $f_k$ joining $p(f_k)$ to $x$ with a uniform bound for the number of nodes and the distance between them. 

If $h_{\gamma_k}$ denotes the holonomy for $\mathbb{P}(F_k)$ associated to $\gamma_{k}$, then we define $$a_{k,x}=h_{\gamma_k}(a(f_k))\quad \text{and}\quad b_{k,x}=h_{\gamma_k}(b(f_k)).$$


When $k$ is big enough, we have the following properties for $a_{k,x}$ and $b_{k,x}$:

\begin{enumerate}[label=(\alph*)] 
\item For every $x\in M$, $a_{k,x}$ and $b_{k,x}$ do not depend on the $su$-path $\gamma_k$.
\item For every $x\in M$, $\mathbb{P}(F_{k,x})(a_{k,x})=a_{k, f_k(x)}$ and $\mathbb{P}(F_{k,x})(b_{k,x})=b_{k, f_k(x)}$.
\item $x\mapsto a_{k,x}$ and $x\mapsto b_{k,x}$ vary continuously with the point $x\in M$. 
\end{enumerate}

Notice that by Corollary \ref{hol a cont} applied to $f_k\to f$ and $p(f_k)\to p$, there exist $C>0$ and a function $\psi$ depending only on $k$ such that for every $su$-path, $\gamma_k$, given by Proposition \ref{uniform} for $f_k$ joining $p(f_k)$ to $x$, there exists an $su$-path for $f$ denoted by $\gamma$ and joining $p$ to a point $y$ close to $x$, such that $$d(h_{\gamma}(a), h_{\gamma_k}(a(f_k)))\leq \psi(k) + C\, d(a, a(f_k)),$$ and $$d(h_{\gamma}(b), h_{\gamma_k}(b(f_k)))\leq \psi(k) + C\, d(b, b(f_k)),$$ where $\psi(k)\to 0$ as $k\to \infty$. 

By Remark \ref{finv}, since $m^{+}$ and $m^{-}$ are $\mathbb{P}(F)$-invariant probability measures and the subspaces associated to the eigenvalues of $\mathit{Df^{n_p}_p\vert E^c(p,f)}$ are denoted by $a,b\in \mathit{\mathbb{P}(E^c(p,f))}$, we have $a=\mathit{supp}\, m^{+}_p$ and $b=\mathit{supp}\, m^{-}_p$. Then, for $k$ big enough, we have 
\begin{equation}\label{ab}
d(\mathit{supp}\, m^{+}_x,h_{\gamma_{k}}(a(f_k)))< \epsilon/2 \quad \text{and} \quad d(\mathit{supp}\, m^{-}_x, h_{\gamma_{k}}(b(f_k)))< \epsilon/2.
\end{equation}
Here we are using that the disintegrations for $m^{+}$ and $m^{-}$ are bi-invariant and $m^{+}_x$ and $m^{-}_x$ depend continuously on the base point $x\in M$.

Moreover, the disintegration of $m_k$ is also bi-invariant. Therefore, $$\mathit{supp}\,m_{k,x}= \{h_{\gamma_k}(a(f_k)), h_{\gamma_k}(b(f_k))\},$$ for every $su$-path $\gamma_k$ joining $p(f_k)$ to $x$. This is a consequence of the fact that we are assuming $m_{k, p(f_k)}= t \delta_{a(f_k)} + (1-t) \delta_{b(f_k)}$. 

We are going to use the observations above to prove properties (a) to (c).

\emph{Proof of} (a): Let $\gamma_{k,1}$ and $\gamma_{k,2}$ be two $su$-paths for $f_k$ joining $p(f_k)$ to $x$ and given by Proposition \ref{uniform}. Since the disintegration of $m_k$ is bi-invariant, $h_{\gamma_{k,1}}(a(f_k))=h_{\gamma_{k,2}}(a(f_k))$ or $h_{\gamma_{k,1}}(a(f_k))=h_{\gamma_{k,2}}(b(f_k))$. Suppose we are in the second case, by Equation (\ref{ab}), $$d(\mathit{supp}\, m^{+}_x, h_{\gamma_{k,1}}(a(f_k)))< \epsilon/2 \quad \text{and} \quad d(\mathit{supp}\, m^{-}_x, h_{\gamma_{k,2}}(b(f_k)))< \epsilon/2,$$ and we get a contradiction. This is because $\epsilon>0$ was chosen to satisfy $$B_{\epsilon}(M^{+})\cap B_{\epsilon}(M^{-})=\emptyset,$$ where $M^{+}=\mathit{supp}\,m^{+}$ and $M^{-}=\mathit{supp}\,m^{-}$. Then, $h_{\gamma_{k,1}}(a(f_k))=h_{\gamma_{k,2}}(a(f_k))$ and $h_{\gamma_{k,1}}(b(f_k))=h_{\gamma_{k,2}}(b(f_k))$ as we wanted to prove. 

\emph{Proof of} (b): By the definition of $a_{k,x}$ and $b_{k,x}$ and the disintegration of $m_k$ being bi-invariant, we have $$\mathit{supp}\, m_{k,x}=\{a_{k,x}, b_{k,x}\}\quad \text{and} \quad \mathit{supp}\, m_{k,f_k(x)}=\{a_{k,f_k(x)}, b_{k,f_k(x)}\},$$ for every $x\in M$. Moreover, by Remark \ref{finv}, $$\mathbb{P}(F_{k,x})(a_{k,x})=a_{k,f_k(x)}\quad \text{or}\quad \mathbb{P}(F_{k,x})(a_{k,x})=b_{k,f_k(x)}.$$
Suppose the second case happens, that is $\mathbb{P}(F_{k,x})(a_{k,x})=b_{k,f_k(x)}$. Then, by Equation (\ref{ab}), $$d(\mathit{supp}\, m^{-}_{f_k(x)},\mathbb{P}(F_{k,x})(a_{k,x}))< \epsilon/4,$$ if $k$ is big enough. Moreover, since $f_k\to f$ and $m^{-}_x$ depends continuously on the base point $x\in M$, we can suppose that $k$ is big enough such that $$d(\mathit{supp}\, m^{-}_{f(x)},\mathbb{P}(F_{k,x})(a_{k,x}))< \epsilon/2.$$
On the other hand, again by Equation (\ref{ab}), $$d(\mathit{supp}\, m^{+}_{x}, a_{k,x})< \epsilon/4,$$ and since $f_k\to f$, we have $$d(\mathit{supp}\, m^{+}_{f(x)},\mathbb{P}(F_{k,x})(a_{k,x}))=d(\mathbb{P}(F_{x})(\mathit{supp}\, m^{+}_{x}),\mathbb{P}(F_{k,x})(a_{k,x}))< \epsilon/2.$$ 
Therefore, $\mathbb{P}(F_{k,x})(a_{k,x})$ is $\epsilon/2$-close to $M^{+}=\mathit{supp}\,m^{+}$ and $M^{-}=\mathit{supp}\,m^{-}$ which is a contradiction because $\epsilon$ satisfies $B_{\epsilon}(M^{+})\cap B_{\epsilon}(M^{-})=\emptyset.$ This implies that $\mathbb{P}(F_{k,x})(a_{k,x})=a_{k,f_k(x)}$ and $\mathbb{P}(F_{k,x})(b_{k,x})=b_{k,f_k(x)}$ for every $x\in M$.

\emph{Proof of} (c): This is a consequence of $m_{k,x}$ depending continuously on the base point $x\in M$, $\mathit{supp}\, m_{k,x}=\{a_{k,x},b_{k,x}\}$, Equation (\ref{ab}) and $B_{\epsilon}(M^{+})\cap B_{\epsilon}(M^{-})=\emptyset.$

Properties (a) to (c) allow us to define two $\mathbb{P}(F_k)$-invariant probability measures projecting down to $\mu$ by $$m_k^{+}=\int \delta_{a_{k,x}}\, d\mu \quad \text{and}\quad m_k^{-}=\int \delta_{b_{k,x}}\, d\mu.$$ 

Recall that we are assuming there exists $K\in \mathbb{N}$ such that for every $k\geq K$ there exists $t\in (0,1)$ such that $m_{k, p(f_k)}= t \delta_{a(f_k)} + (1-t) \delta_{b(f_k)}$. Moreover, by the definition of $a_{k,x}$ and $b_{k,x}$, we have $m_{k,x}= t \delta_{a_{k,x}} + (1-t) \delta_{b_{k,x}}$. Then, $m_k$ can be written as $m_k= t\, m_k^{+} + (1-t)\, m_k^{-}$. This is a contradiction, since we chose $m_k$ to be ergodic. Therefore, there exists a subsequence $k_j$ such that $m_{k_j,p(f_{k_j})}=\delta_{a(f_{k_j})}$ or $m_{k_j,p(f_{k_j})}=\delta_{b(f_{k_j})}$ for every $j\in \mathbb{N}$. 

Suppose $m_{k_j,p(f_{k_j})}=\delta_{a(f_{k_j})}$ for every $j\in \mathbb{N}$, the other case is analogous. By definition of $a_{k,x}$, we have $m_{k_j,x}=\delta_{a_{k_j}, x}$ for every $x\in M$. By Equation (\ref{ab}), for every $\epsilon>0$ there exists $J\in \mathbb{N}$ such that $d(a_{{k_j},x}, \mathit{supp}\, m^{+}_x)< \epsilon$ for every $j\geq J$ and every $x\in M$. Then, $m_{k_j}\to m^{+}$ when $j$ goes to $\infty$.

Since $f_k\to f$, we have $\lambda_{+}(F_{k_j})=\int{ \Phi_{k_j}(x,v)\, d\, m_{k_j}}\to \int{ \Phi(x,v)\, d\, m^{+}}=\lambda_{+}(F)$. However, we were assuming that $\lambda^c_1(f_k)=\lambda^c_2(f_k)$ for every $k\in \mathbb{N}$ and $\lambda^c_1(f)\neq \lambda^c_2(f)$. By Equation (\ref{lyap0}), these hypotheses implies $\lambda_{+}(F_k)=0$ for every $k\in \mathbb{N}$ and $\lambda_{+}(F)>0$. Therefore, the conclusion we obtain, $\lambda_{+}(F_{k_j})\to \lambda_{+}(F)$ for some subsequence $k_j$, is a contradiction. Finally, we conclude $\lambda^c_1(f)$ must be equal to $\lambda^c_2(f)$ as we wanted to prove. 
\qed

\section{Proof of Theorem B and D}

Recall that $\omega$ denotes a symplectic form, $\mu$ denotes a probability measure in the Lebesgue class and by Definition \ref{gral} if $r\geq 2$ and $*\in \{\mu, \omega\}$, then $\mathcal{B}^r_{*}(M)$ denotes the subset of $\mathit{PH}^r_{*}(M)$ where $f\in \mathcal{B}^r_{*}(M)$ if $f$ is accessible, $\alpha$-pinched and $\alpha$-bunched for some $\alpha>0$ and its center bundle is 2-dimensional. Moreover, by Definition \ref{pinch} we say that a periodic point $p$ is a pinching hyperbolic periodic point if the eigenvalues of $Df^{n_p}_p|E^c(p)$ have different norms and both norms are different from one.

In the proof of Theorem A, we observed that if $f\in \mathcal{B}^r_{\omega}(M)$ is a discontinuity point for the center Lyapunov exponents, then $f$ has a pinching hyperbolic periodic point. 

We are going to prove the following theorem which implies Theorem B and D. 

\begin{teo}\label{interior} Let $r\geq 2$ and $*\in \{\mu, \omega\}$. Suppose $f\in \mathcal{B}^r_{*}(M)$ and $f$ has a pinching hyperbolic periodic point $p$. If $f$ is a $C^r$ discontinuity point for the center Lyapunov exponents, then $f$ can be $C^r$-approximated by open sets of $C^r$ continuity points for the center Lyapunov exponents. 
\end{teo}

As we mentioned before, the proof of this theorem is a consequence of Proposition \ref{discont}, Theorem \ref{invariant} and the arguments in \cite{Ma}.

Let $r\geq 2$ and $*\in \{\mu, \omega\}$. Assume that $f\in \mathcal{B}^r_{*}(M)$, $f$ has a pinching hyperbolic point and $f$ is a $C^r$ discontinuity point for the center Lyapunov exponents.

Let $p$ be a pinching hyperbolic periodic point for $f$. Define $a,b\in \mathit{\mathbb{P}(E^c(p,f))}$ as $a=\mathbb{P}(E^1)$ and $b=\mathbb{P}(E^2)$, where $E^1$ and $E^2$ are the subspaces of $\mathit{E^c(p,f)}$ associated to the eigenvalues of $\mathit{Df^{n_p}_p\vert E^c(p,f)}$.

If we consider the cocycle over $f$ given by $\mathbb{P}(F)$ where $\mathit{F=\eta\cdot Df\vert E^c}$ (see Section 3), then by Proposition \ref{discont} and Theorem \ref{invariant}, the probability measure $m^{+}$, defined in Proposition \ref{nozero}, admits a disintegration $\{m^{+}_x: x\in M\}$ which is bi-invariant and its conditional probabilities depend continuously on the base point $x\in M$, relative to the weak$^{*}$-topology.

In the following toy model, we explain the main ideas and steps in the proof of Theorem \ref{interior}. These ideas are classical and have already appeared, for example, in \cite{AV1,VY}. 

\subsection*{Toy model} Suppose there exists $z\in M$ such that $z\in W^{ss}(p)\cap W^{uu}(p)$. Then, the disintegration of $m^{+}$, $\{m^{+}_x : x\in M\}$ satisfies $$(h^s_{p,z}(f))_{*} m^{+}_p=m^{+}_z\quad \text{and} \quad(h^u_{p,z}(f))_{*} m^{+}_p=m^{+}_z.$$ If $\mathit{supp}\, m^{+}_p=a$, we have that there exists $c\in \mathit{\mathbb{P}(E^c(z))}$ such that $h^s_{p,z}(f)(a)=h^u_{p,z}(f)(a)=c$. 

Since $p$ is periodic and $z$ has no recurrence, there exists $\delta>0$ such that $f^j(z)\notin B_{\delta}(z)$ for every $j\in \mathbb{Z}\setminus \{0\}$ and $f^j(p)\notin B_{\delta}(z)$ for every $j\in \mathbb{Z}$. We are going to consider a perturbation of $f$ supported in $B_{\delta}(z)$ That is, a diffeomorphism which is $C^r$-close enough to $f$ and such that $g(x)=f(x)$ if $x\notin B_{\delta}(z)$. This perturbation is chosen in order to have, $h^s_{p,z}(g)=R_{\beta}\circ h^s_{p,z}(f)$ and $h^u_{p,z}(g)=h^u_{p,z}(f)$. Here, $R_{\beta}$ denotes a rotation of angle $\beta>0$. Recall $\mathbb{P}(F)$ is a cocycle of circle diffeomorphisms over $f$ and then it makes sense to consider rotations in $\mathit{\mathbb{P}(E^c(p,f))}$. 

If $g$ is a $C^r$ discontinuity point for the center Lyapunov exponents, then we have that the measure $m^{+}_g$, given by Proposition \ref{nozero}, admits a disintegration which is bi-invariant and its conditional probabilities depend continuously on the base point $x\in M$. Since $a=\mathit{supp}\, m^{+}_{g,p}$, then $h^s_{p,z}(g)(a)=R_{\beta}(c)$ and $h^u_{p,z}(g)(a)=c$. This is a contradiction and therefore $g$ has to be a $C^r$ continuity point for the center Lyapunov exponents. 

\subsection*{Strategy of the proof} We will use the same argument than in \cite{M} to generalize the ideas in the toy model.

First we find an $su$-path from $p$ to itself with a special node $z$, which is slowly accumulated by the orbits of all the nodes including its own. This is done in Proposition \ref{pivot}. Next, we construct a sequence of $C^r$-perturbations denoted by $f_k$ and supported in $B_{\delta_k}(z)$. The details are given in Lemma \ref{pert}.

Then, we study how the $su$-path and the holonomies change under the variation of the diffeomorphism. The main results are Proposition \ref{nodes} and Proposition \ref{pert-holonomy}. The main observation is that the variation in the holonomies is exponentially small in $k$, although the size of the perturbations $\delta_k$ is polynomial in $k$. This will allow us to break the rigidity given by the existence of continuous sections which are invariant by $\mathbb{P}(F)$ and by the invariant stable and unstable holonomies of $\mathbb{P}(F)$.

We are going to suppose that $f_k$ is a $C^r$ discontinuity point for the center Lyapunov exponents for every $k\in \mathbb{N}$. Therefore, by Proposition \ref{discont} and Theorem \ref{invariant}, for every $k\in \mathbb{N}$ we have a family of disintegrations $\{m^{+}_{k,x} : x\in M\}$ associated to the measure $m^{+}_k$ given by Proposition \ref{nozero}. Moreover, $m^{+}_{k,x}$ depends continuously on the base point $x\in M$. In order to conclude the argument we need the functions $x\mapsto m^{+}_{k,x}$ to be H\"older continuous. We are not able to prove this property, but the problem is solved using the hyperbolicity of $p$ and Proposition \ref{continuity-su}.

Moreover, we prove that for every $k\in \mathbb{N}$ there exist a neighborhood of $f_k$ where all the above are still valid. This will allow us to conclude that $f$ can be $C^r$-approximated by open sets of $C^r$ continuity points for the center Lyapunov exponents. 

In order to simplify the presentation we state here the results of \cite{M} that we use in the proof of the theorem.  
  
\vspace{0.3cm}
\emph{Proof of Theorem \ref{interior}.}

Let $r\geq 2$ and $*\in \{\mu, \omega\}$. Assume that $f\in \mathcal{B}^r_{*}(M)$, $f$ has a pinching hyperbolic point and $f$ is a  $C^r$ discontinuity point for the center Lyapunov exponents.

Let $p$ be a pinching hyperbolic periodic point for $f$. Consider the cocycle $\mathbb{P}(F)$ over $f$ given by $\mathit{F=\eta\cdot Df\vert E^c}$ and the probability measure $m^{+}$ defined in Proposition \ref{nozero}. By Proposition \ref{discont} and Theorem \ref{invariant}, $m^{+}$ admits a disintegration which is bi-invariant and its conditional probabilities depend continuously on the base point $x\in M$, relative to the weak$^{*}$-topology. Suppose $a=\mathit{supp}\, m^{+}_p$.

In the toy model we assume there exists $z\in W^{ss}(p)\cap W^{uu}(p)$, therefore we have an $su$-path for $f$ given by $\zeta=[p,z,p]$. The next proposition will allow us to define an $su$-path for $f$ from $p$ to itself which will generalize the situation we consider in the toy model. 

\begin{prop}[Proposition 8.2 in \cite{ASV}]\label{pivot} Let $f$ be a $C^2$ partially hyperbolic accessible diffeomorphism. Then, for every $x\in M$ there exist an $su$-path, $\zeta=[z_0, ..., z_N]$ with $x=z_0=z_N$, $l\in \{0,...,N\}$ and $c>0$ such that $$dist (f^j(z_i), z_l)\geq \frac{c}{1+j^2},$$ for every $(j,i)\in \mathbb{Z}\times \{0,..., N\}\setminus (0,l)$.
\end{prop}

Let $\zeta=[z_0, ..., z_N]$ be the $su$-path given by Proposition \ref{pivot} for $f$ and $p$. We are going to find a sequence of perturbations supported around $z_l$. 

In order to guarantee exponential estimations in Proposition \ref{pert-holonomy}, Equation (\ref{orbit}) and Equation (\ref{hypothesis3}) we need to consider a technical constant $\sigma>0$. The value of $\sigma$ depends only on $f$ and we fix it from now on. More precisely, $\sigma=\sigma(\upsilon, \alpha, n_p, N)$, where $\upsilon$ represents the functions in Equation (\ref{ph}) for $f$, $\alpha$ is the exponent for which $f$ is $\alpha$-pinched and $\alpha$-bunched, $n_p$ is the period of $p$ and $N$ the number of nodes in the $su$-path given by Proposition \ref{pivot}.

Define, 
\begin{equation}\label{deltak}
\delta_k=\frac{c}{1+(\sigma\,k)^2},
\end{equation}
for every $k\geq1$, where $c>0$ is given by Proposition \ref{pivot}.  

\begin{lemma}[Lemma 4.4 in \cite{M}]\label{pert} Let $r\geq 2$ and $*\in \{\mu, \omega\}$. There exist $\epsilon_0>0$, $k_0\in \mathbb{N}$ and $C_0>0$ such that for any $0<\epsilon<\epsilon_0$ and $k\geq k_0$, there exists $f_k\in \mathcal{B}^r_{*}(M)$ which is $\epsilon$ $C^r$-close to $f$ and such that
\begin{enumerate}[label=\emph{(\alph*)}]
\item $f_k(x)=f(x)$ if $x\notin B_{\delta_k}(z_l)$,
\item $f_k(z_l)=f(z_l)$ and
\item $Df_k(z_l)=Df(z_l) \circ A_{\beta_k}$ where $\sin\, \beta_k=C_0\, \delta_k^{r-1} \epsilon$ and $A_{\beta_k}$ is the linear map from $TM_{z_l}$ to $TM_{z_l}$ given in coordinates $TM=E^s\oplus E^c\oplus E^u$ by $$\begin{pmatrix}
Id_s & 0 & 0 \\
0  & R_{\beta_k} & 0 \\
0 & 0 & Id_u
\end{pmatrix} $$ with $Id_{**}\colon E^{**}_{z}\to E^{**}_{z}$ being the identity map for $**\in \{s,u\}$ and $R_{\beta_k}$ the counterclockwise rotation of angle $\beta_k$ in some (symplectic) base $\{e_1,f_1\}$ of $E^c(z)$. 
\end{enumerate}
Moreover, if we fix $\epsilon>0$ and consider the sequence $\{f_k\}_{k\geq k_0}$, then $f_k\to f$ in the $C^1$ topology when $k\to \infty$. 
\end{lemma}

The properties of the sequence $f_k$ given by Lemma \ref{pert} allow us to understand how the dynamics is changing. The next proposition studies how the $su$-path $\zeta$ given by Proposition \ref{pivot} varies with $f$. 

\begin{prop}[Proposition 4.8 in \cite{M}]\label{nodes} If $\zeta=[z_0,...,z_N]$ is the $su$-path given by Proposition \ref{pivot} for $f$ and $p$ and $f_k$ is given by Lemma \ref{pert} for some $\epsilon>0$, then there exist $C_1>0$, $\tau\in (0,1)$ and $k_1\in \mathbb{N}$ such that for every $k\geq k_1$ there exists an $su$-path for $f_k$, $\zeta (f_k)=[z_{0}(f_k),...,z_{N}(f_k)]$, with $z_{0}(f_k)=z_0=p$ and such that $$dist(z_i, z_{i}(f_k))< C_1\, \tau^{\sigma_1 k},$$ for every $i\in \{1,..., N\}$, where $\sigma_1=\sigma\,\alpha^{N}$.
\end{prop} 
Although $\zeta$ is a closed $su$-path, $z_0=z_N=p$, the $su$-path for $f_k$ given by this proposition is not necessarily closed. We can have $z_{0}(f_k)\neq z_{N}(f_k)$.

Fix some $\epsilon>0$ small enough and let $f_k$ be given by Lemma \ref{pert} for this $\epsilon$. Then, in the notation of Proposition \ref{nodes}, define $$\zeta_1=[z_0,...,z_l],\quad \quad \zeta_2=[z_N,..., z_l],$$ $$\zeta_1(f_k)=[p,...,z_l(f_k)]\quad \text{and} \quad \zeta_2(f_k)=[z_N(f_k),...,z_l(f_k)].$$  

Analogous to Definition \ref{hol}, we can define invariant stable and unstable holono\-mies for the linear cocycle $\mathit{F'=Df\vert E^c}$. If $f\in \mathcal{B}^r_{*}(M)$, by Section 3 of \cite{ASV}, there exist invariant stable and invariant unstable holonomies associated to $F'$. We are going to denote these holonomies by $H^s_{x,y}$ and $H^u_{x,y}$. Moreover, if $h^s_{x,y}$ and $h^u_{x,y}$ denote the holonomies for $\mathbb{P}(F')$, then $h^s_{x,y}=\mathbb{P}(H^s_{x,y})$ and $h^u_{x,y}=\mathbb{P}(H^u_{x,y})$. Recall $\mathbb{P}(F')=\mathbb{P}(F)$ if $\mathit{F=\eta\cdot F'}$.

For $i\in \{1,2\}$, $H_{\zeta_i}$ will denote the holonomy for $F'$ defined by $\zeta_i$ and $H_{\zeta_i(f_k)}$ the holonomy for $F'_k=Df_k\vert E^c(f_k)$ defined by $\zeta_i(f_k)$. Then, $$H_{\zeta_1}\colon \mathit{E^c(p)}\to \mathit{E^c(z_l, f)}, \quad \quad H_{\zeta_2}\colon \mathit{E^c(p)}\to \mathit{E^c(z_l, f)},$$ $$H_{\zeta_1(f_k)}\colon \mathit{E^c(p)}\to \mathit{E^c(z_l(f_k), f_k)}\quad \text{and} \quad H_{\zeta_2(f_k)}\colon \mathit{E^c(z_N(f_k), f_k)}\to \mathit{E^c(z_l(f_k), f_k)}.$$

We can suppose $z_{l-1}\in W^s(z_l)$ and $z_l\in W^u(z_{l+1})$.

\begin{prop}[Corollary 4.10 in \cite{M}]\label{pert-holonomy} There exist $C>0$, $\lambda\in (0,1)$ and $K\in \mathbb{N}$ such that for every $k\geq K$, $c\in \mathit{E^c(p)}$ and $c_k\in \mathit{E^c(z_N(f_k),f_k)}$ we have
\begin{equation}\label{rotation}
d( R_{\beta_k}^{-1}\circ H_{\zeta_1}(c), H_{\zeta_1(f_k)}(c))\leq C\, \lambda^k,
\end{equation}
and
\begin{equation}\label{direct}
d(H_{\zeta_2}(c), H_{\zeta_2(f_k)}(c_k))\leq C\, \lambda^k + C\, d(c, c_k),
\end{equation}
where $R_{\beta_k}\colon \mathit{E^c(z_l,f)}\to \mathit{E^c(z_l,f)}$ is the rotation of angle $\beta_k>0$ defined by Lemma \ref{pert}.
\end{prop}

Here the distance between points in different fibers of $TM$ is defined, as before, using parallel transport. More precisely, for every $x,y\in M$ close enough, denote $\pi_{x,y}:T_{x}M\longrightarrow T_{y}M$ the parallel transport along $\chi$, where $\chi$ is the geodesic satisfying $dist(x,y)=\text{length}(\chi)$. Then, given two points $(x,v)$ and $(y,w)$ in $TM$ define $$d((x,v),(y,w))=dist(x,y) + \left\|\pi_{x,y}(v)- w\right\|.$$

Suppose $f_k$ is a $C^r$ discontinuity point for the center Lyapunov exponents for every $k$, then we can define $m_{f_k}^{+}$ by Proposition \ref{nozero} and it will admit a disintegration which is bi-invariant and its conditional probabilities depend continuously on the base point $x\in M$. This is a consequence of Proposition \ref{discont} and Theorem \ref{invariant}.

Let $z_N(f_k)$ be given by Proposition \ref{nodes} and $$q_k=f_k^{-n_p k}(z_N(f_k)).$$ Then, there exists $C_2>1$ such that 
\begin{equation}\label{orbit}
dist(p,q_k)\leq C_2^{n_p k} dist(p, z_N(f_k)) \leq C_1\, (C_2^{n_p} \tau^{\sigma_1})^{k}.
\end{equation} Here $C_2$ depends on the functions in Equation (\ref{ph}) and $\sigma_1=\sigma\,\alpha^{N}$. The constant $\sigma$ was chosen in order to have this expression going to zero as $k\to \infty$. 

Now, we consider the $su$-paths given by Proposition \ref{continuity-su} for $f_k\to f$, $x=y=p$, $x_k=p$ and $y_k=q_k$ for every $k$. Therefore, there exist a subsequence, that we will continue to denote $k$ to simplify the notation, $su$-paths for $f_k$ denoted by $\gamma_k$ joining $p$ to $q_k$ and an $su$-path for $f$ denoted by $\gamma$ joining $p$ to itself. 

Moreover, by Corollary \ref{bound-sequence}, there exists $C>0$ such that $$d(h_{\gamma}(a), h_{\gamma_{k}}(a))\leq \psi(k), $$ where $\psi(k)\to 0$ as $k\to \infty$. Since the disintegrations of $m^{+}$ and $m^{+}_{f_k}$ are bi-invariant and $a=\mathit{supp}\, m^{+}_p=\mathit{supp}\, m^{+}_{f_k,p}$, we conclude that for every $\epsilon>0$ there exists $K_1\in \mathbb{N}$ such that for every $k\geq K_1$ we have,
\begin{equation}\label{cerca}
d(\mathit{supp}\, m^{+}_{f_k,q_k}, a)< \epsilon.
\end{equation} 

Since $p$ is a pinching hyperbolic periodic point, there exist $C_3>0$, $\theta_0>0$, $\rho_0\in (0,1)$ such that for the subspaces $E^1$, $E^2$ of $\mathit{E^c(p)}$ associated to the eigenvalues of $\mathit{Df^{n_p}_p\vert E^c(p)}$ we have: 

For every $F^1$ and $F^2$ one-dimensional subspaces of $E^c(p)$ with 
\begin{equation}\label{hypothesis2}
\mathit{\max\{ \angle(E^1,F^1), \angle(E^2, F^2)\}}<\theta_0,
\end{equation}
then
\begin{equation}\label{unstable2}
\angle( Df^{n_pj}(E^1), Df^{n_pj}(F^1))\leq C_3\, \rho_0^j
\end{equation} 
 and
\begin{equation}
\angle( Df^{-n_pj}(E^2), Df^{-n_pj}(F^2))\leq C_3\, \rho_0^j,  
\end{equation} 
for every $j\geq 0$.

Recall $q_k=f_k^{-n_p k}(z_N(f_k)).$ Let $d_k=\mathit{supp}\, m^{+}_{f_k,q_k}$ and $c_k=\mathbb{P}(F^{n_p}_k(q_k)) (d_k).$ Then, by Remark \ref{finv}, $c_k=\mathit{supp}\, m^{+}_{f_k, z_N(f_k)}$ and by Equation (\ref{cerca}), $dist(d_k,a)<\theta_0$ if $k$ is big enough. 

Therefore, there exists $C_4>0$ such that
\begin{equation}\label{hypothesis3}
\begin{aligned}
d(a,c_k)&=d(\mathbb{P}(F^{n_p k}(p))(a), \mathbb{P}(F^{n_p k}_k(q_k))(d_k))\\
        &\leq C_1\; \rho^k + C_1\, (C_4^{n_p}\tau^{\sigma_1})^k.\\
\end{aligned}
\end{equation} 

In order to obtain the estimations in Equation (\ref{hypothesis3}) we use triangular inequality to get two terms: the first one is bounded using Equation (\ref{unstable2}) and the second one using that $\mathbb{P}(F_k)$ is Lipschitz. When $k$ is big enough, we can take the constant $C_4$ to be uniform for every $k$, depending only on $f$. The constant $\sigma$ was chosen to have the expression on the second term going to zero exponentially fast as $k\to \infty$.

Therefore, there exist $\widehat{C}>0$ and $\rho\in (0,1)$ such that for every $k$ big enough,
\begin{equation}\label{support}
d(a, \mathit{supp}\, m^{+}_{f_k,z_N(f_k)})< \widehat{C}\, \rho^k.
\end{equation} 

Summarizing, by Equation (\ref{rotation}), Equation (\ref{direct}) and Equation (\ref{support}) we have the following estimations.

There exist $C>0$, $\lambda\in (0,1)$ and $K\in \mathbb{N}$ such that the following equations are valid for every $k\geq K$:
\begin{equation}\label{rotation-fin}
d( R_{\beta_k}^{-1}\circ H_{\zeta_1}(a), H_{\zeta_1(f_k)}(a))\leq C\, \lambda^k,
\end{equation}
and
\begin{equation}\label{direct-fin}
d(H_{\zeta_2}(a), H_{\zeta_2(f_k)}(c_k))\leq C\, \lambda^k,
\end{equation}
where $R_{\beta_k}\colon \mathit{E^c(z_l,f)}\to \mathit{E^c(z_l,f)}$ is the rotation of angle $\beta_k>0$ defined by Lemma \ref{pert}, $a=\mathit{supp}\, m^{+}_p=\mathit{supp}\,m^{+}_{f_k,p}$ and $c_k=\mathit{supp}\, m^{+}_{f_k,z_N(f_k)}$. 

For $i\in \{1,2\}$, let $h_{\zeta_i(f_k)}$ denote the holonomy for $\mathbb{P}(F_k)$ defined by $\zeta_i(f_k)$. Then, $h_{\zeta_i(f_k)}=\mathbb{P}(H_{\zeta_i(f_k)})$. 

By Equation (\ref{deltak}) and Lemma \ref{pert}, $\frac{\lambda^k}{\sin^2 \beta_{k}}\rightarrow 0$ when $k\to \infty.$ This is true because $\beta_k$ depends polynomially on $k$. This property, together with Equation (\ref{rotation-fin}) and Equation (\ref{direct-fin}) imply that for $k$ big enough $$h_{\zeta_1(f_k)}(\mathit{supp}\, m^{+}_{f_k,p})\neq h_{\zeta_2(f_k)}(\mathit{supp}\, m^{+}_{f_k,z_N(f_k)}).$$

This is a contradiction because we were assuming that every $f_k$ was a $C^r$ discontinuity point for the center Lyapunov exponents and therefore there existed a disintegration for every $m^{+}_{f_k}$ which was bi-invariant. Therefore, we have proved that every $f\in \mathcal{B}^r_{*}(M)$ having a pinching hyperbolic periodic point can be $C^r$- approximated by diffeomorphisms which are $C^r$  continuity points for the center Lyapunov exponents. 

In order to conclude the proof of Theorem \ref{interior}, we need to prove that the argument above can be carry in a robust way. 

Fix $k$ big enough. If $g$ is a diffeomorphism close enough to $f_k$, then there exist $p(g)$ a pinching hyperbolic periodic point of $g$ with $n_p=per(p(g))$ which is close to $p$. Consider $a(g), b(g)\in \mathit{\mathbb{P}(E^c(p(g),g))}$ the projectivization of the two subspaces of $\mathit{E^c(p(g),g)}$ which are invariant by $\mathit{Dg^{n_p}_{p(g)}}$. We can suppose $a(g)$ is close to $a\in \mathit{\mathbb{P}(E^c(p,f_k))}$. Recall $a=\mathit{supp}\, m^{+}_p$.

Let $\zeta_k=[z_{0}(f_k),...,z_{N}(f_k)]$ be the $su$-path for $f_k$ given by Proposition \ref{nodes}. Then, for every $\epsilon>0$ there exists a $C^1$ neighborhood of $f_k$, $\mathcal{W}(f_k)$, such that if $g\in \mathcal{W}(f_k)$, then there exists an $su$-path for $g$, denoted by $\zeta(g)=[z_0(g),...,z_N(g)]$, with $z_0(g)=p(g)$ and such that $dist(z_i(f_k), z_i(g))<\epsilon$ for every $i\in \{0,...,N\}$.

Let $$\zeta_1(g)=[z_0(g),...,z_l(g)]\quad \text{and} \quad \zeta_2(g)=[z_N(g),...,z_l(g)].$$ For $i\in \{1,2\}$ we denote by $H_{\zeta_i(g)}$ the holonomy for $\mathit{G=Dg\vert E^c(g)}$ defined by $\zeta_i(g)$. Then, $$H_{\zeta_1(g)}\colon \mathit{E^c(p(g),g)}\to \mathit{E^c(z_l(g), g)}\quad \text{and} \quad H_{\zeta_2(g)}\colon \mathit{E^c(z_N(g), g)}\to \mathit{E^c(z_l(g), g)}.$$

If $\lambda\in (0,1)$ and $K\in \mathbb{N}$ are given by Proposition \ref{pert-holonomy}, then for every $k\geq K$, we will define a $C^r$ neighborhood of $f_k$, $\mathcal{V}_k(f_k)\subset \mathcal{B}^r_{*}(M)$. 

We say that $g\in \mathcal{V}_k(f_k)$ if it satisfies the following: 
\begin{enumerate}[label=(\alph*)]
\item $dist_{C^r}(g,f_k)<\epsilon$,
\item $dist_{C^1}(g,f_k)<\lambda^k$,
\item $d(a(g),a)< \lambda^k$,
\item the distance between the nodes of $\zeta(g)$ and $\zeta(f_k)$ is bounded by $\lambda^k$,
\item $d(H_{\zeta_1(g)}(c(g)), H_{\zeta_1(f_k)}(c))\leq \lambda^k + C\, d(c, c_k)$ for every $c\in \mathit{E^c(p,f_k)}$ and every $c(g)\in \mathit{E^c(p(g),g)}$,
\item $d(H_{\zeta_2(g)}(c_{k}(g)), H_{\zeta_2(f_k)}(c_k))\leq \lambda^k + C\, d(c, c_k)$ for every $c_k\in \mathit{E^c(z_N(f_k),f_k)}$ and every $c_{k}(g)\in \mathit{E^c(z_N(g),g)}$ and
\item if $q(g)=g^{-n_p k}(z_N(g))$, then $dist(q(g), q_k)< \lambda^k$. 
\end{enumerate}

The existence of a neighborhood of $f_k$ where properties (e) and (f) hold is a consequence of Proposition 3.4 and Corollary 3.5 of \cite{M}.

The proof of Theorem \ref{interior} will follow if we prove that there exists $m\in \mathbb{N}$ such that every $g\in \mathcal{V}_m(f_m)$ is a $C^r$ continuity point for the center Lyapunov exponents. In order to do that, we suppose that for every $k\geq K$ there exists $g_k\in \mathcal{V}_k(f_k)$ such that $g_k$ is a $C^r$ discontinuity point for the center Lyapunov exponents. The properties (a)-(g) were chosen to allow us to extend the same argument that we used above for $f_k$, now for the sequence $g_k$. Therefore, from the assumption of $g_k$ being a $C^r$ discontinuity point for the center Lyapunov exponents for every $k$, we arrive to a contradiction. This implies that there exists $m\in \mathbb{N}$ such that every $g\in \mathcal{V}_m(f_m)$ is a $C^r$ continuity point for the center Lyapunov exponents as we wanted to prove.
\qed

\end{document}